\documentclass[12pt]{article}
\usepackage{etex}
\usepackage[all]{xy}
\usepackage{cite}
\usepackage{amsfonts}
\usepackage{amsthm}
\usepackage{enumerate}
\usepackage{graphicx}
\usepackage{mathrsfs}
\usepackage{bm}
\usepackage{amssymb,amsmath} 
\usepackage{makecell}
\usepackage{tikz}
\usepackage{pgf}
\usepackage{tikz}
\usetikzlibrary{patterns}
\usepackage{pgffor}
\usepackage{pgfcalendar}
\usepackage{pgfpages}
\usepackage{shuffle,yfonts}
\usepackage{mathtools}
\DeclareFontFamily{U}{shuffle}{}
\DeclareFontShape{U}{shuffle}{m}{n}{ <-8>shuffle7 <8->shuffle10}{}

\newcommand{\nc}{\newcommand}
\nc{\AMZV}{\mathsf {AMZV}}
\nc{\ud}{\mathrm{d}}
\nc{\ES}{\mathsf {ES}}
\nc{\MZV}{\mathsf {MZV}}
\nc{\MtV}{\mathsf {MtV}}
\nc{\MTV}{\mathsf {MTV}}
\nc{\MSV}{\mathsf {MSV}}
\nc{\MMV}{\mathsf {MMV}}
\nc{\MMVo}{\mathsf {MMVo}}
\nc{\MMVe}{\mathsf {MMVe}}
\nc{\AMMV}{\mathsf {AMMV}}
\nc{\AMTV}{\mathsf {AMTV}}
\nc{\AMtV}{\mathsf {AMtV}}
\nc{\AMSV}{\mathsf {AMSV}}
\nc{\CMZV}{\mathsf {CMZV}}
\nc{\sha}{\shuffle}
\nc{\cst}{\rotatebox[origin=c]{180}{$\sha$}}
\nc{\cstt}{\rotatebox[origin=c]{180}{$\scriptstyle \sha$}}
\nc{\de}{\delta}
\nc{\DD}{{\mathbb D}}
\nc{\anbb}[1]{\left\langle#1\right\rangle}
\nc{\bibb}[1]{\left\{#1\right\}}
\nc{\mibb}[1]{\left[#1\right]}
\nc{\smbb}[1]{\left(#1\right)}
\nc{\doubb}[1]{\llbracket#1\rrbracket}
\nc{\dm}[1]{\left|#1\right|}

\nc{\Gbinom}[2]{\genfrac{(}{)}{0mm}{0}{#1}{#2}}
\nc{\gbinom}[2]{\genfrac{(}{)}{0mm}{1}{#1}{#2}}
\nc{\Rbinom}[2]{\genfrac{\langle}{\rangle}{0mm}{0}{#1}{#2}}
\nc{\rbinom}[2]{\genfrac{\langle}{\rangle}{0mm}{1}{#1}{#2}}
\nc{\Qbinom}[2]{\genfrac{[}{]}{0mm}{0}{#1}{#2}_q}
\nc{\qbinom}[2]{\genfrac{[}{]}{0mm}{1}{#1}{#2}_q}

\nc{\binq}[2]{\genfrac{[}{]}{0mm}{0}{#1}{#2}}
\nc{\tbnq}[2]{\genfrac{[}{]}{0mm}{1}{#1}{#2}}
\nc{\cinq}[2]{\genfrac{\{}{\}}{0mm}{0}{#1}{#2}}
\nc{\tcnq}[2]{\genfrac{\{}{\}}{0mm}{1}{#1}{#2}}

\nc{\mfrac}[2]{\genfrac{}{}{0pt}{}{#1}{#2}}
\nc{\tf}{\tfrac}
\nc{\db}{{\mathbb D}}
\nc{\pari}{{\rm par}}
\nc{\dk}{{\mathbb K}}
\nc{\ola}{\overleftarrow}
\nc{\ora}{\overrightarrow}
\nc{\lra}{\longrightarrow}
\nc{\Lra}{\Longrightarrow}
\nc\Res{{\rm Res}}
\nc\setX{{\mathsf{X}}}
\nc\fA{{\mathfrak{A}}}
\nc\evaM{{\texttt{M}}}
\nc\evaML{{\text{\em{\texttt{M}}}}}
\nc\z{{\texttt{z}}}
\nc\emz{\emph{\texttt{z}}}
\nc\tx{{\texttt{x}}}
\nc\txp{{\tx_1}} 
\nc\txn{{\tx_{-1}}} 
\nc\neo{{1}}
\nc{\yi}{{1}}
\nc\one{{-1}}
\nc\gD{{\Delta}}
\nc\eps{{\varepsilon}}
\nc{\bfMB}{{\bf MB}}
\nc{\bftB}{{\bf tB}}
\nc{\bfTB}{{\bf TB}}
\nc{\bfSB}{{\bf SB}}
\nc{\bfB}{{\bf B}}
\nc{\bfp}{{\bf p}}
\nc{\bfq}{{\bf q}}
\nc{\bfr}{{\bf r}}
\nc{\bfu}{{\bf u}}
\nc{\bfv}{{\bf v}}
\nc{\bfw}{{\bf w}}
\nc{\bfy}{{\bf y}}
\nc{\T}{\ddot{t}}
\nc{\bfe}{{\boldsymbol{\sl{e}}}}
\nc{\bfi}{{\boldsymbol{\sl{i}}}}
\nc{\bfj}{{\boldsymbol{\sl{j}}}}
\nc{\bfk}{{\boldsymbol{\sl{k}}}}
\nc{\bfl}{{\boldsymbol{\sl{l}}}}
\nc{\bfm}{{\boldsymbol{\sl{m}}}}
\nc{\bfn}{{\boldsymbol{\sl{n}}}}
\nc{\bfs}{{\boldsymbol{\sl{s}}}}
\nc{\bft}{{\boldsymbol{\sl{t}}}}
\nc{\bfx}{{\boldsymbol{\sl{x}}}}
\nc{\bfz}{{\boldsymbol{\sl{z}}}}
\nc\bfgs{{\boldsymbol \gs}}
\nc\bfgl{{\boldsymbol \lambda}}
\nc\bfsi{{\boldsymbol \gs}}
\nc\bfet{{\boldsymbol \eta}}
\nc\bfeta{{\boldsymbol \eta}}
\nc\bfeps{{\boldsymbol \eps}}
\nc\mmu{{\boldsymbol \mu}}
\nc\bfone{{\bf 1}}
\nc{\myone}{{1}}

 \nc{\calA}{{\mathcal A}}
 \nc{\calB}{{\mathcal B}}
 \nc{\calC}{{\mathcal C}}
 \nc{\calD}{{\mathcal D}}
 \nc{\calE}{{\mathcal E}}
 \nc{\calF}{{\mathcal F}}
 \nc{\calG}{{\mathcal G}}
 \nc{\calH}{{\mathcal H}}
 \nc{\calI}{{\mathcal I}}
 \nc{\calJ}{{\mathcal J}}
 \nc{\calK}{{\mathcal K}}
 \nc{\calL}{{\mathcal L}}
 \nc{\calM}{{\mathcal M}}
 \nc{\calN}{{\mathcal N}}
 \nc{\calO}{{\mathcal O}}
 \nc{\calP}{{\mathcal P}}
 \nc{\calQ}{{\mathcal Q}}
 \nc{\calR}{{\mathcal R}}
 \nc{\calS}{{\mathcal S}}
 \nc{\calT}{{\mathcal T}}
 \nc{\calU}{{\mathcal U}}
 \nc{\calV}{{\mathcal V}}
 \nc{\calW}{{\mathcal W}}
 \nc{\calX}{{\mathcal X}}
 \nc{\calY}{{\mathcal Y}}
 \nc{\calZ}{{\mathcal Z}}
  \nc{\cala}{{\mathcal a}}
 \nc{\calb}{{\mathcal b}}
 \nc{\calc}{{\mathcal c}}
 \nc{\cald}{{\mathcal d}}
 \nc{\cale}{{\mathcal e}}
 \nc{\calf}{{\mathcal f}}
 \nc{\calg}{{\mathcal g}}
 \nc{\calh}{{\mathcal h}}
 \nc{\cali}{{\mathcal i}}
 \nc{\calj}{{\mathcal j}}
 \nc{\calk}{{\mathcal k}}
 \nc{\call}{{\mathcal l}}
 \nc{\calm}{{\mathcal m}}
 \nc{\caln}{{\mathcal n}}
 \nc{\calo}{{\mathcal o}}
 \nc{\calp}{{\mathsf p}}
 \nc{\calq}{{\mathcal q}}
 \nc{\calr}{{\mathcal r}}
 \nc{\cals}{{\mathcal s}}
 \nc{\calt}{{\mathcal t}}
 \nc{\calu}{{\mathcal u}}
 \nc{\calv}{{\mathcal v}}
 \nc{\calw}{{\mathcal w}}
 \nc{\calx}{{\mathcal x}}
 \nc{\caly}{{\mathcal y}}
 \nc{\calz}{{\mathcal z}}
 \nc{\ot}{{\otimes}}
\usetikzlibrary{arrows,shapes,chains}
 \allowdisplaybreaks

\catcode`!=11
\let\!int\int \def\int{\displaystyle\!int}
\let\!lim\lim \def\lim{\displaystyle\!lim}
\let\!sum\sum \def\sum{\displaystyle\!sum}
\let\!sup\sup \def\sup{\displaystyle\!sup}
\let\!inf\inf \def\inf{\displaystyle\!inf}
\let\!cap\cap \def\cap{\displaystyle\!cap}
\let\!max\max \def\max{\displaystyle\!max}
\let\!min\min \def\min{\displaystyle\!min}
\let\!frac\frac \def\frac{\displaystyle\!frac}
\catcode`!=12

\nc{\gam}{{\gamma}}
\nc{\gG}{{\Gamma}}
\nc{\om}{{\omega}}
\nc{\vep}{{\varepsilon}}
\nc{\ga}{{\alpha}}
\nc{\gl}{{\lambda}}
\nc{\gb}{{\beta}}
\nc{\gd}{{\delta}}
\nc{\gf}{{\varphi}}
\nc{\gs}{{\sigma}}
\nc{\gk}{{\kappa}}
\nc{\gS}{\Sigma}
\let\oldsection\section
\renewcommand\section{\setcounter{equation}{0}\oldsection}

\allowdisplaybreaks

\DeclareMathOperator*{\dep}{dep}

\DeclareMathOperator{\dch}{dch}

\nc\UU{\mbox{\bfseries U}}
\nc\FF{\mbox{\bfseries \itshape F}}
\nc\h{\mbox{\bfseries \itshape h}}\nc\dd{\mbox{d}}
\nc\g{\mbox{\bfseries \itshape g}}
\nc\xx{\mbox{\bfseries \itshape x}}

\def\N{\mathbb{N}}
\def\Z{\mathbb{Z}}
\def\Q{\mathbb{Q}}

\def\xx{\left(\frac{1-x}{1+x} \right)}

\def\ol{\overline}
\nc\divg{{\text{div}}}
\theoremstyle{plain}
\newtheorem{thm}{Theorem}[section]
\newtheorem{lem}[thm]{Lemma}

\newtheorem{conj}[thm]{Conjecture}
\newtheorem{prop}[thm]{Proposition}
\theoremstyle{definition}

\newtheorem{re}[thm]{Remark}

\nc{\cicc}[1]{{}_{{}^{ \bigcirc\hskip-1.2ex{#1}\hskip.3ex{}}}}
\nc{\cic}[1]{{}^{\bigcirc\hskip-1.15ex{\raisebox{-0.015cm}{\text{$\scriptscriptstyle #1$}}}\hskip.25ex{}}}
\nc{\ccic}[1]{{}^{\bigcirc\hskip-1.5ex{\raisebox{-0.015cm}{\text{$\scriptscriptstyle #1$}}}\hskip.25ex{}}}
\nc{\ncic}[1]{ {\bigcirc\hskip-1.6ex{\raisebox{-0.0cm}{\text{$\scriptstyle #1$}}}\hskip.25ex{}}}
\nc{\nncic}[1]{ {\bigcirc\hskip-2ex{\raisebox{-0.0cm}{\text{$\scriptstyle #1$}}}\hskip.25ex{}}}
\nc{\cci}[1]{{}_{{}^{ {\textstyle \bigcirc}\hskip-2.05ex{#1}\hskip-.35ex{}}}}
\nc{\ccicc}[1]{{}_{{}^{ {\textstyle \bigcirc}\hskip-1.55ex{#1}\hskip-0.1ex{}}}}
\nc{\x}{\rm{x}}
\nc{\tworow}[2]{\left(#1 \atop #2\right)}

\nc{\fl}{{\mathfrak l}}
\nc{\fm}{{\mathfrak m}}

\begin{document}
\title{\bf On Some Unramified Families of Motivic Euler Sums}
\author{{Ce Xu${}^{a,}$\thanks{Email: cexu2020@ahnu.edu.cn} \quad and \quad Jianqiang Zhao${}^{b,}$\thanks{Email: zhaoj@ihes.fr}}\\[1mm]
\small a. School of Mathematics and Statistics, Anhui Normal University,  Wuhu 241002, P.R. China\\
\small b. Department of Mathematics, The Bishop's School, La Jolla,  CA 92037, USA}

\date{}
\maketitle
\noindent{\bf Abstract.} It is well known that sometimes Euler sums (i.e., alternating multiple zeta values) can be expressed as $\Q$-linear combinations of multiple zeta values (MZVs). In her thesis Glanois presented a criterion for motivic Euler sums to be unramified, namely, expressible as $\Q$-linear combinations of motivic MZVs. By applying this criterion we present a few families of such unramified motivic Euler sums in two groups. In one such group we can further prove the concrete identities relating the motivic Euler sums to the motivic MZVs, determined up to rational multiple of a motivic Riemann zeta value by a result of Brown, under the assumption that the analytic version of such identities hold.

\medskip \noindent{\bf Keywords}: (motivic) multiple zeta values, (motivic) Euler sums.

\medskip \noindent{\bf AMS Subject Classifications (2020):} 11M32; 11M99.

\section{Introduction}
The ubiquitous nature of multiple zeta values (MZVs) has attracted many mathematicians and theoretical physicists in recent years after the seminal works of Zagier \cite{DZ1994} and Hoffman \cite{H1992}. Its higher level generalization is given by the \emph{colored multiple zeta values} (CMZVs) of level $N$ defined as follows. Let $\N$ be the set of positive integers and $\N_0=\N\cup\{0\}$.
For any composition $(s_1,\dots,s_d)\in\N^d$ and $N$-th roots of unity $(\eps_1,\dots,\eps_d)$ we define
\begin{equation*}
\zeta\tworow{s_1,\dots,s_d}{\eps_1,\dots,\eps_d}:=\sum_{0<k_1<\dots<k_d} \frac{\eps_1^{k_1}\cdots\eps_d^{k_d}}{k_1^{s_1}\cdots k_d^{s_d}}.
\end{equation*}
To guarantee convergence we impose the condition that $(s_d,\eps_d)\ne (1,1)$.

The CMZVs have played a pivotal role in the theory of mixed Tate motives over $\Z[\mu_N][1/N]$ (resp. $\Z[\mu_N]$), where $\mu_N=\exp(2\pi i/N)$, for $N=1,2,4,6,8$ (resp. $N=6$) as manifested by the works \cite{Brown2012,Deligne2010,Glanois2015}. In fact, they first appeared unexpectedly in the computation of Feynman integrals in the 1990s. In particular, the level two MZVs, sometimes also called \emph{Euler sums}, have been studied quite intensively in \cite{BlumleinBrVe2010,Broadhurst1996a,Glanois2016,JinLi2018}. To save space, if $\eps_j=-1$ then we conventionally put a bar on top of $s_j$. If such an Euler sum can be expressed in terms of the MZVs then we say it is an \emph{unramified} Euler sum because the corresponding motivic version is unramified. These values are also called honorary MZVs by Broadhurst.  For example, we have the following beautiful unramified family of Euler sums discovered numerically in \cite{BorweinBrBr1997} and proved by the second author \cite{Zhao2010a}
\begin{equation}\label{equ:12bar}
  8^\ell   \zeta(\{1,\bar2\}_\ell)=\zeta(3_\ell)
\end{equation}
for all $\ell\in\N$. Here and in the rest of the paper, $\bfs_n$ means the string $\bfs$ is repeated $n$ times and if the string has only one number then we remove the curly brackets to save space.

Our goal in this paper is to lift \eqref{equ:12bar} to its motivic version and then provide a few more families of unramified 
motivic Euler sums in two groups. Let $\zeta^\fm(\bfs)$ be the motivic version of the Euler sum $\zeta(\bfs)$
and $\zeta_a^\fm(\bfs)$ the shuffle regularized motivic Euler sum for all $a\in\N$
(see next section for precise definition).  Then the first group is presented below.

\begin{thm} \label{thm:2bar3Motivic}
For all integers $a, \ell\ge 0$ the motivic Euler sums 
\begin{align}\label{equ:2bar3original}
&\zeta_a^\fm(\{\bar2,3\}_\ell), \qquad \zeta_a^\fm(\{\bar2,3\}_\ell,\bar2),  \\
&\zeta_a^\fm(\{3,\bar2\}_\ell), \qquad \zeta_a^\fm(\{3,\bar2\}_\ell,3), \label{equ:13bar2original}
\end{align}
are all unramified.
\end{thm}

Using exactly the same approach we may verify the second group of unramified families
which have the the following form.

\begin{thm} \label{thm:bar21Motivic}
For all integers $a, \ell\ge 0$ the motivic Euler sums 
\begin{equation*}
\zeta_a^\fm(\{\bar2,1\}_\ell), \quad \zeta_a^\fm(\{\bar2,1\}_\ell,\bar2),   \quad\zeta_a^\fm(\{1,\bar2\}_\ell), \quad \zeta_a^\fm(\{1,\bar2\}_\ell,1),
\end{equation*}
are all unramified.
\end{thm}

\begin{re}
In \cite{Glanois2015,Glanois2016} Glanois proved that an interpolated version of motivic Euler sums denoted by $\zeta^{\sharp,\fm}(\bfs)$
are all unramified if $\bfs=(\{\ol{\text{even}},\text{odd}\}_\ell)$ where the even components can differ and so can the odd ones.
This is unlikely true for the motivic Euler sums in general. Numerical computation using the LLL algorithm shows that it is unlikely that
$\zeta(\bar2,3,\bar2,5)$ is in the weight 12 piece of the MZV space.
\end{re}

The main results obtained in this paper can be regarded as some good evidence for the following conjecture.

\begin{conj}
Let $a\in\N_0$, $\ell\in\N$, and $u_j,v_j\in X:=\{1,\bar2,3,\bar4,\dots\}$ for $j=1,\dots,\ell$ such that
$1\not\in \{u_j: 1\le j\le \ell\}\cap \{v_j: 1\le j\le \ell\}$. Then
\begin{equation*}
\sum_{\gs,\tau\in{\mathfrak S}_\ell} \zeta_a^\fm(u_{\gs(1)},v_{\tau(1)},\dots,u_{\gs(\ell)},v_{\tau(\ell)})
\end{equation*}
and
\begin{equation*}
\sum_{\gs\in{\mathfrak S}_\ell,\tau\in{\mathfrak S}_{\ell-1}} \zeta_a^\fm(u_{\gs(1)},v_{\tau(1)},\dots,v_{\tau(\ell-1)},u_{\gs(\ell)}).
\end{equation*}
are both unramified, where ${\mathfrak S}_\ell$ is the symmetry group of $\ell$ letters.
\end{conj}

When $a=0$, after taking the period map this conjecture is reduced to a conjecture first proposed by M. Hirose and N. Sato \cite{SH2019}.

We will prove Theorem \ref{thm:2bar3Motivic} in section \ref{sec:proofOfThm:2bar3Motivic}. The idea can be applied similarly
to prove Theorem \ref{thm:bar21Motivic}. Instead of repeating this we would like to offer a precise version of it.

\begin{conj}\label{conj:bar21bar2}
For all $\ell\in\N$, we have
\begin{align*}
2^{3\ell+1}\zeta(\{\bar2,1\}_\ell,\bar2)=&\, -\sum_{\ga+\gb=\ell} (-1)^\ga \zeta(3_{\ga},2,3_{\gb}), \\
2^{3\ell-1}\zeta_\sha(\{1,\bar2\}_\ell,1)=&\, -3\sum_{\ga+\gb=\ell, 2\nmid \beta}\zeta(3_{\ga},1,3_{\gb}), \\
2^{3\ell}\zeta_\sha(\{\bar2,1\}_\ell)=&\, \zeta(3_\ell)-2\sum_{\ga+\gb=\ell-1} (-1)^{\ga} \zeta_1(3_{\ga},2,3_{\gb})
\end{align*}
where $\zeta_\sha(\bfs)$ is the shuffle regularized value of $\zeta(\bfs)$ (see \cite{Racinet2002} or \cite[Ch.~13]{ZhaoBook}).
\end{conj}
The first two identities were numerically discovered by  M. Hirose and N. Sato \cite{SH2019}.

\begin{thm} \label{thm:bar21bar2Motivic}
Assume Conjecture~\ref{conj:bar21bar2} holds. Then for all $\ell\in\N$, we have
\begin{align}\label{equ:1bar2Motivic}
2^{3\ell}  \zeta^\fm(\{1,\bar2\}_\ell)=&\, \zeta^\fm(3_\ell),\\
2^{3\ell+1}\zeta^\fm(\{\bar2,1\}_\ell,\bar2)=&\,
    -\sum_{\ga+\gb=\ell} (-1)^\ga\zeta^\fm(3_{\ga},2,3_{\gb}),\label{equ:bar21bar2Motivic}\\
2^{3\ell}\zeta^\fm(\{\bar2,1\}_\ell)=&\,
    \zeta^\fm(3_\ell)-2\sum_{\ga+\gb=\ell-1} (-1)^{\ga} \zeta_1^\fm(3_{\ga},2,3_{\gb}), \label{equ:bar21Motivic}\\
2^{3\ell-1}\zeta^\fm(\{1,\bar2\}_\ell,1)=&\,-3\sum_{\ga+\gb=\ell, 2\nmid \beta}\zeta^\fm(3_{\ga},1,3_{\gb}).\label{equ:12bar1Motivic}
\end{align}
\end{thm}

The proof of Theorem \ref{thm:bar21bar2Motivic} relies on the descent criterion first shown by Glanois \cite{Glanois2015,Glanois2016},
which can tell us when an Euler sum is unramified (i.e., lying in the MZV space). The computation also implies Theorem~\ref{thm:bar21Motivic}
without assuming Conjecture~\ref{conj:bar21bar2} since we only need to show the coaction \eqref{equ:coaction} is stable. 
But in order to prove the precise expressions of these motivic Euler sums
in terms of motivic MZVs as listed in \eqref{equ:bar21bar2Motivic}--\eqref{equ:12bar1Motivic} 
we need to use their analytic version given by the conjecture to 
show the expressions still hold after the coaction is applied.

\section{Motivic set-up}
We adopt the motivic set-up defined by \cite{Glanois2016}. Let $\Gamma_{N}$ be the group of $N$-th unity for any $N\in\N$. Let $w\in\N$ and suppose $a_j=0$ or $a_j\in\Gamma_{N}$ for all $0\le j\le w+1$. The most important property of the motivic integrals is that there is a period map $\dch$  such that

\begin{equation}\label{equ:periodMap}
\dch \big(I^\fm(a_0;a_1,\dots,a_w;a_{w+1})\big)=(-1)^{\dep(a_1,\dots,a_w)}
\int_{a_0}^{a_{w+1}} \frac{dt}{t-a_1} \cdots  \frac{dt}{t-a_w}
\end{equation}
as an iterated integral (we always integrate from left to right in this paper)
whenever it converges. Here $\dep(a_1,\dots,a_w):=\sharp\{1\le j\le w: a_j\ne 0\}$ is called
the \emph{depth} and $w$ the \emph{weight} of the integral (and its corresponding MZV).
When the integral diverges then there is a way to regularize it so
that it becomes a polynomial in $\CMZV[T]$, where $\CMZV$ is the $\Q$-span of all CMZVs of level $N$.
By setting $T=0$ we can define the so-called shuffle regularized CMZVs (see \cite[Ch.~13]{ZhaoBook} for more details).

We now list all the important properties of the motivic integrals as follows (cf. \cite[\S2.4]{Brown2012}, \cite[p.~9]{Glanois2016}, and \cite[\S2, (I1)-(I6)]{Murakami2021}):
\begin{itemize}
	\item[(I1)] Empty word: $I^\fm(a_{0}; \emptyset; a_{1})=1$.

	\item[(I2)] Trivial path: $\forall n\ge1$, $I^\fm(a_{0}; a_{1}, \dots, a_{n}; a_{n+1})=0$ if $a_{0}=a_{n+1}$.

	\item[(I3)] Shuffle product: for any $s_{1}, \dots, s_n\in\mathbb N$ and
$\eps_{1}, \dots,\eps_n\in\Gamma_N$, define
\begin{equation*}
\zeta_k^\fm\left(s_{1}, \dots, s_n \atop \eps_{1}, \dots,\eps_n \right)
:=I^\fm\big(0; 0^k,\eta_1,
0^{s_1-1},\eta_2,
0^{s_2-1}, \dots, \eta_n,
0^{s_n-1}; 1 \big)
\end{equation*}
where $\eta_j=1/\eps_j\cdots \eps_d$ for all $j=1,\dots,d$. Then
	\begin{equation*}
\zeta_{k}^\fm \left( {s_{1}, \cdots , s_{p} \atop \eps_{1}, \cdots ,\eps_{p} }\right)=
(-1)^{k}\sum_{\substack{i_{1}+ \cdots + i_{p}=k\\ i_{1},\dots,i_{p}\ge0}} \left(\prod_{j=1}^p\binom {s_j+i_j-1} {i_j}\right) \zeta^\fm \left( {s_{1}+i_{1}, \cdots , s_{p}+i_{p} \atop \eps_{1}\ \ , \cdots ,\ \ \eps_{p} }\right).
 \end{equation*}

	\item[(I4)] Regularization: If $a_{1}=\cdots=a_n\in \{0,1\}$, then
	$I^\fm(0; a_{1}, \cdots, a_{n}; 1)=0.$

	\item[(I5)] Path reversal: $I^\fm(a_{0}; a_{1}, \cdots, a_{n}; a_{n+1})= (-1)^n I^\fm(a_{n+1}; a_{n}, \cdots, a_{1}; a_{0}).$

	\item[(I6)] Homothety: $\forall \alpha \in \Gamma_{N}, I^\fm(0; \alpha a_{1}, \cdots, \alpha a_{n}; \alpha a_{n+1})  = I^\fm(0; a_{1}, \cdots, a_{n}; a_{n+1})$.

	\item[(I7)] Change of variable $t\to 1-t: \forall a_{1}, \cdots, a_{n}\in\{0, 1\},$ $$I^\fm(0; a_{1}, \cdots, a_{n}; 1)=  I^\fm(0;1-a_n, \cdots, 1-a_1; 1).$$

	\item[(I8)] Path composition: $\forall a,b, x\in \Gamma_{N} \cup \left\{0\right\}$,
	$$  I^\fm(a; a_{1}, \cdots, a_{n}; b)=\sum_{i=0}^{n} I^\fm(a; a_{1}, \cdots, a_{i}; x) I^\fm(x; a_{i+1}, \cdots, a_{n}; b) .$$
\end{itemize}

Let $\calH^N$ be the $\Q$-vector space spanned by the motivic CMZVs of the form
\begin{equation*}
\zeta^\fm_a\tworow{n_1,\dots,n_d}{\eps_1,\dots,\eps_d}:= I^\fm(0;0_a,\eta_1,0_{n_1-1},\dots,\eta_d,0_{n_d-1};1),
\quad a\in\N_0,
\end{equation*}
where $n_j\in\N, \eps_j\in\Gamma_{N}$, and $\eta_j$'s are defined in (I3) for all $j$.
Note that $a+n_1+\dots+n_d$ is its weight which is denoted by $|\bfn|$ when $a=0$.
If $\eps_j=1$ (resp. $\eps_j=\pm1$) for all $j$ then we get the motivic MZV (resp. Euler sum) $\zeta^\fm_a(n_1,\dots,n_d)$.
For $N=1,2$, set $\calA^N=\calH^N/\zeta^\fm(2)\calH^N$. Denote by
$\calH^N_w$ and $\calA_w^N$ the weight $w$ part for all $w\ge 0$. Let $\calL^N=\calA^N_{>0}/\calA^N_{>0}\cdot \calA^N_{>0}$.
For any weight $w$ and odd $r$ such that $r<w$ one can define a derivation as part of a coaction
\begin{equation}\label{equ:coaction}
D_r: \calH_w^N   \to \calL_r^N \ot \calH_{w-r}^N
\end{equation}
by sending $I^\fm(a_0;a_1,\dots,a_w;a_{w+1})$ to
$$
\sum_{p=0}^{w-r} I^\fl(a_p; a_{p+1},\dots,a_{p+r};a_{p+r+1})\ot I^\fm(a_0;a_1,\dots,a_p,a_{p+r+1},\dots,a_w;a_{w+1}).
$$
The sequence in the left motivic integral is called a \emph{subsequence} of $(a_0;a_1,\dots,a_w;a_{w+1})$ while
that in the right factor is called a \emph{quotient sequence}. Each such a choice is called a \emph{cut}.

Now we set $N=2$ and recall the following theorem which combines
Glanois's result \cite[Cor.\ 2.4]{Glanois2016} with Brown's \cite[Thm.\ 3.3]{Brown2012}. Set $\calH=\calH^1$ and $\calL=\calL^1$.

\begin{thm}\label{thm-Glanois}
Let $a\in\N_0$ and $\bfs$ be a composition of positive integers such that $a+|\bfs|=k$.
Then the weight $k$ motivic Euler sum $\zeta_a^\fm(\bfs)\in\calH_k$ if and only $D_1 \zeta_a^\fm(\bfs)=0$ and
$D_r \zeta_a^\fm(\bfs)\in \calL_r\ot \calH_{k-r}$ for all odd $r<k$.
Moreover, if $D_r \zeta_a^\fm(\bfs)=0$ for all odd $r<k$ then
$\zeta_a^\fm(\bfs)=c \, \zeta^\fm(k)$ for some rational number $c$.
\end{thm}

Set $\eta_j=\prod_{i=j}^d\eps_i$ for all $j=1,\dots,d$ as above (note that $\eta_j=1/\eta_j$) and define
\begin{equation*}
\rho\tworow{n_1,\dots,n_d}{\eps_1,\dots,\eps_d}:=(0;\eta_1,0_{n_1-1},\dots,\eta_d,0_{n_d-1};1).
\end{equation*}
We often use the bar notation instead of the two-row one for Euler sums.

The following simple fact was already noticed by Glanois \cite{Glanois2015} without proof.
We provide one here for completeness.
\begin{lem} \label{lem:D1}
If $\bar1$ does not appear in $\bfs$ then $D_1\zeta_a^\fm(\bfs)=0$ for all $a\in\N_0$.
\end{lem}
\begin{proof}
If $|s_j|>1$ for all components of $\bfs$ then no two consecutive nonzero terms can appear
in $\rho(\bfs)$ so that $D_1=0$ clearly. If $s_j=1$ then the corresponding sign $\eps_j=1$ and therefore $\eta_j=\eta_{j+1}$.
Namely, whenever two consecutive nonzero terms  appear they must have the same sign.
Thus if there are more than one nonzero numbers in the left factor $I^\fl(a;b;c)$
of $D_1$ then it must look like $I^\fl(0;\eta;\eta)$, or $I^\fl(\eta;\eta;0)$,or $I^\fl(\eta;0;\eta)$ or $I^\fl(\eta;\eta;\eta)$
all of which must be 0. All the other cases are clearly 0, too.
\end{proof}

Recall that the Bernoulli numbers $B_n$ are defined by the generating function
$$
\frac{t}{e^t-1}=\sum_{n=0}^\infty B_n\frac{t^n}{n!}.
$$

\begin{prop} \label{prop-singleES}
For any positive integers $m\ge 2$ and $d\ge 1$ we have
\begin{equation}\label{equ:zbarmd}
    \zeta^\fm(\{\bar m\}_d)\in \calH_{md}.
\end{equation}
Moreover, if $m=2s$ is even we have
\begin{equation}\label{equ:2sd}
    \zeta^\fm(\{\ol{2s}\}_d)= c\, \zeta^\fm(2sd)
\end{equation}
where
$$
c= \frac{-2 (2sd)!}{16^{sd} B_{2sd} } \sum_{\substack{\gS_{j=1}^s n_j=sd\\ n_j \in\N_0\ \forall j}}
\left(\prod_{j=1}^s \frac{\big(1+e^{\tfrac{\pi i}{2s}}\big)^{2n_j+1}+\big(1-e^{\tfrac{\pi i}{2s}}\big)^{2n_j+1} }{2(2n_j+1)!}
e^{\tfrac{\pi i}{s}\underset{\scriptscriptstyle j=1}{\overset{\scriptscriptstyle m}{\gS}} (2j-1)n_j}
\right)\in\Q.
$$
\end{prop}

\begin{proof}
When $d=1$ we see that $D_r\zeta^\fm(\bar m)=0$ for all $r<m$ since all cuts
for $D_r$ either have 0 on both ends or have all 0's inside. Thus  \eqref{equ:2sd}
follows directly from Theorem \ref{thm-Glanois}.

Suppose now $d>1$. To save space we let $m=n+1$ and use $\bar1$ to denote $-1$ . Then
\begin{equation*}
    \zeta^\fm(\{\bar m\}_{2\ell})=I^\fm(0;\{1,0_n,\bar1,0_n\}_\ell;1),\quad
   \zeta^\fm(\{\bar m\}_{2\ell+1})=I^\fm(0;\bar1,0_n,\{1,0_n,\bar1,0_n\}_\ell;1).
\end{equation*}
For every odd integer $r< m$ we clearly have $D_r\zeta^\fm(\{\bar m\}_d)=0$.

Let $r\ge m$ be an odd integer. Suppose first that $m=2s$ is an even integer. If $d=2\ell$ is even
let $r=m+i+2mp<md$ for some integer $i,p\ge 0$. Then all nonzero terms of $D_r$ can be paired as shown below:
\begin{center}
\begin{tikzpicture}[scale=0.9]
\node (A0) at (0.05,0) {$0;\{1,0_n,\bar1,0_n\}_j,1,0_{n-i-1},0,0_i,\{\bar1,0_n,1,0_n\}_p,\bar1,0_i,0,0_{n-i-1},1,0_n,,\bar1,0_n,\{1,0_n,\bar1,0_n\}_k;$};
\draw (8.7,0) node {\raisebox{0.1ex}{$1$}};
\node (A) at (-1.6,0.5) {${}$};
\draw (-5.1,0.25) to (-5.1,0.5) to (A) node {$\cic{1}$} to (1.6,0.5)  to (1.6,0.25);
\node (B) at (0.45,-0.5) {${}$};
\draw (-3.2,-0.25) to (-3.2,-0.5) to (B) node {$\cic{2}$} to (3.6,-0.5)  to (3.6,-0.25);
\node (A) at (-0.5,-1.2) {Possible cuts of $D_r\zeta^\fm(\{\ol{n+1}\}_d), 2j+2p+2k+4=d$ is even.};
\end{tikzpicture}
\end{center}
Strictly speaking there are also cut pairs such as $\ncic{\bar1}+\ncic{\bar2}{}$\,
for which 1 and $\bar1$ are exchanged inside $\ncic{1}\,${} and $\ncic{2}\,${} in the above graph.
Every such pair have the same quotient sequence but reversed subsequences of each other.
Thus the pair cancel each other by path reversal (I5).
The argument for odd $d=2\ell+1$ is exactly the same. Thus, by Theorem \ref{thm-Glanois} we see that
$\zeta^\fm(\{\ol{2s}\}_d)=c\zeta^\fm(2sd)$ for some rational number $c$. Applying
the period map \eqref{equ:periodMap} we can verify \eqref{equ:2sd} by \cite[Thm.\ 2]{ShenHe2021}.

Suppose now that $m$ is an odd integer. Then the above proof still works unless $r=mp$ for some odd $p<d$.
Now, if a cut starts at a 0 immediately before a $\pm 1$ then the term will be
canceled by the cut starting from the next $\pm1$, unless the cut itself ends with the last 1
which yields the following single nonzero term
\begin{equation*}
D_r \zeta^\fm(\{\bar m\}_d)=\zeta^\fm(\{\bar m\}_p)\ot \zeta^\fm(\{\bar m\}_{d-p})\in\calL_{mp}\ot \calH_{m(d-p)}
\end{equation*}
by induction. Hence \eqref{equ:zbarmd} follows by Theorem \ref{thm-Glanois}. This completes the proof
of the proposition.
\end{proof}

\section{Proof of \eqref{equ:1bar2Motivic} in Theorem \ref{thm:bar21bar2Motivic} }
In this section, we first prove \eqref{equ:1bar2Motivic} due to its simplicity and repeated use throughout the paper.
The proof technique is also very typical for this kind of claims on unramified motivic Euler sums.
As pointed out by S.\  Charlton, equation \eqref{equ:1bar2Motivic} follows from the original proof of \eqref{equ:12bar}
in \cite{Zhao2010a} which uses the double shuffle relations and the distribution relations all of which are motivic.
For convenience we restate  \eqref{equ:1bar2Motivic} as follows: for all $\ell\in\N$
\begin{equation}\label{equ:1bar2MotivicAgain}
  8^\ell \zeta^\fm(\{1,\bar2\}_\ell)=\zeta^\fm(3_\ell).
\end{equation}
We now prove by \eqref{equ:1bar2MotivicAgain} by induction on $\ell$.

\subsection{The base case}
If $\ell=1$ then we see easily that $D_1 \zeta^\fm(1,\bar2)=0$ by Lemma~\ref{lem:D1}.
Thus by Theorem~\ref{thm-Glanois} $\zeta^\fm(1,\bar2)=c\zeta^\fm(3)$ for some $c\in \Q$.
Applying the period map \eqref{equ:periodMap} and comparing with \eqref{equ:12bar} we get $c=1/8$ immediately.

We now assume that \eqref{equ:1bar2MotivicAgain} holds for all positive integers $\ell<k$. We need to show that
\begin{equation}\label{equ:D_r1bar2}
 8^k D_r \zeta^\fm(\{1,\bar2\}_k)=D_r \zeta^\fm(3_k)
\end{equation}
for all positive integers $r=6n+3, 6n+5, 6n+7<3k$ where $n\in\Z$.

\subsection{Inductive step}

\subsubsection{$r=6n+3$}
We have the following picture when $k$ is even (odd $k$ case can be dealt with similarly).
\begin{center}
\begin{tikzpicture}[scale=0.9]
\node (A1) at (-2.6,0) {$0;$};
\node (A2) at (-2.2,0) {$1,$};
\node (A3) at (-1.8,0) {$1,$};
\draw[dashed] (-1.6,0.7) to (-1.6,-0.7);
\node (A4) at (-1.4,0) {$0,$};
\node (A5) at (-1,0) {$\bar1,$};
\node (A6) at (-0.6,0) {$\bar1,$};
\draw[dashed] (-0.2,0.7) to (-0.2,-0.7);
\node (A7) at (0.2,0) {$\cdots\!,$};
\draw[dashed] (0.6,0.7) to (0.6,-0.7);
\node (A7) at (0.85,0) {$0,$};
\node (A7) at (1.25,0) {$1,$};
\node (A9) at (1.65,0) {$1,$};
\node (A9) at (2.8,0.04) {$\{0\bar1\bar1011\}_n$};
\node (A10) at (4.05,0) {$0,$};
\node (A11) at (4.45,0) {$\bar1,$};
\node (A11) at (4.85,0) {$\bar1,$};
\node (A12) at (5.25,0) {$0,$};
\node (A13) at (5.65,0) {$1,$};
\node (A13) at (6.45,0) {$1,\cdots\!,$};
\node (A14) at (7.35,0) {$0,$};
\node (A14) at (7.75,0) {$\bar1,$};
\node (A14) at (8.15,0) {$\bar1,$};
\node (A15) at (8.55,0) {$0;$};
\node (A16) at (8.95,0.03) {$1$};

\node (C1) at (2.6,-0.4) {${}$};
\node (D4) at (3,0.6) {${}$};
\node (D3) at (3.4,-0.6) {${}$};
\draw (0.8,-0.25) to (0.8,-0.4) to (C1) node {$\cic{1}$} to (4.4,-0.4)  to (4.4,-0.325);
\draw (1.2,0.25) to (1.2,0.6) to (D4) node {$\cic{2}$} to (4.8,0.6)  to (4.8,0.25);
\draw (1.6,-0.25) to (1.6,-0.6) to (D3) node {$\cic{3}$} to (5.2,-0.6)  to (5.2,-0.25);
\draw[dashed] (1.8,0.7) to (1.8,-0.7);
\draw[dashed] (3.8,0.7) to (3.8,-0.7);
\draw[dashed] (6.2,0.7) to (6.2,-0.7);
\draw[dashed] (5,0.7) to (5,-0.7);
\draw[dashed] (7.1,0.7) to (7.1,-0.7);
\draw[dashed] (8.3,0.7) to (8.3,-0.7);
\node (lab) at (2.8,-1.2) {Possible cuts of $D_{6n+3} \zeta^\fm(\{1,\bar2\}_k)$};
\end{tikzpicture}
\end{center}

Suppose a cut starts in a 011-block. It is easy to see that $\ncic{1}+\ncic{3}=0$ since they have opposite subsequences but the same quotient one. Further, the subsequence is anti-symmetric so that $\ncic{2}=0$ by path reversal (I5) and homothety with $\alpha=-1$ in (I6). If  a cut starts in a $0\bar1\bar1$-block then the same argument works. However, if a cut of type $\ncic{\bar1}${} ends at the last 1 then it cannot be canceled. Thus
\begin{equation*}
8^k D_{6n+3} \zeta^\fm(\{1,\bar2\}_k)= 8^k \zeta^\fl(\{1,\bar2\}_{2n+1})\ot \zeta^\fm(\{1,\bar2\}_{k-2n-1})
=\zeta^\fl(3_{2n+1})\ot \zeta^\fm(3_{k-2n-1})
\end{equation*}
by induction. Similarly, by the picture
\begin{center}
\begin{tikzpicture}[scale=0.9]
\node (A4) at (-1.4,0) {$0;$};
\node (A5) at (-1,0) {$1,$};
\node (A6) at (-0.6,0) {$0,$};
\draw[dashed] (-0.3,0.7) to (-0.3,-0.7);
\node (A7) at (0.1,0) {$\cdots\!,$};
\draw[dashed] (0.5,0.7) to (0.5,-0.7);
\node (A7) at (0.75,0) {$0,$};
\node (A7) at (1.15,0) {$1,$};
\node (A9) at (1.55,0) {$0,$};
\node (A9) at (2.8,0.04)  {$\{010010\}_n$};
\node (A10) at (4.05,0) {$0,$};
\node (A11) at (4.45,0) {$1,$};
\node (A11) at (4.85,0) {$0,$};
\node (A12) at (5.25,0) {$0,$};
\node (A13) at (5.65,0) {$1,$};
\node (A13) at (6.45,0) {$0,\cdots\!,$};
\node (A14) at (7.35,0) {$0,$};
\node (A14) at (7.75,0) {$1,$};
\node (A14) at (8.15,0) {$0,$};
\node (A15) at (8.55,0) {$0;$};
\node (A16) at (8.95,.03) {$1$};
\node (C1) at (2.6,-0.6) {${}$};
\node (D4) at (3,0.6) {${}$};
\draw (0.7,-0.25) to (0.7,-0.6) to (C1) node {$\cic{1}$} to (4.4,-0.6)  to (4.4,-0.25);
\draw (1.1,0.25) to (1.1,0.6) to (D4) node {$\cic{2}$} to (4.8,0.6)  to (4.8,0.25);
\draw[dashed] (1.8,0.7) to (1.8,-0.7);
\draw[dashed] (3.8,0.7) to (3.8,-0.7);
\draw[dashed] (5.05,0.7) to (5.05,-0.7);
\draw[dashed] (6.25,0.7) to (6.25,-0.7);
\draw[dashed] (7.1,0.7) to (7.1,-0.7);
\draw[dashed] (8.35,0.7) to (8.35,-0.7);
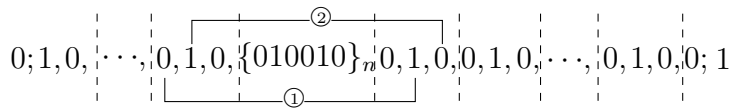
\node (lab) at (3,-1.2) {Possible cuts of $D_{6n+3} \zeta^\fm(3_k)$};
\end{tikzpicture}
\end{center}
we see that $\ncic{1}+\ncic{2}=0$ unless the cut $\ncic{1}\,${} ends at the last 1 so that it cannot be canceled. Thus
\eqref{equ:D_r1bar2} holds for $r=6n+3$.

\subsubsection{$r=6n+5$}
It is obvious that $D_{6n+5}=0$ on both sides of \eqref{equ:D_r1bar2} since both $\rho(3_k)$ and $\rho(\{1,\bar2\}_k)$
have period 6 so that every cut of $D_{6n+5}$ starts and ends with the same number.

\subsubsection{$r=6n+7$}
We have the following pictures of possible cuts.

\begin{equation} \label{fig:r=6n+7-bar21}
\text{
\begin{tikzpicture}[scale=0.9]
\node (A1) at (-2.6,0) {$0;$};
\node (A2) at (-2.2,0) {$1,$};
\node (A3) at (-1.8,0) {$1,$};
\draw[dashed] (-1.6,0.7) to (-1.6,-0.7);
\node (A4) at (-1.4,0) {$0,$};
\node (A5) at (-1,0) {$\bar1,$};
\node (A6) at (-0.6,0) {$\bar1,$};
\draw[dashed] (-0.3,0.7) to (-0.3,-0.7);
\node (A7) at (0.1,0) {$\cdots\!,$};
\draw[dashed] (0.5,0.7) to (0.5,-0.7);
\node (A7) at (0.75,0) {$0,$};
\node (A7) at (1.15,0) {$1,$};
\node (A9) at (1.55,0) {$1,$};
\node (A9) at (2.8,0.04) {$\{0\bar1\bar1011\}_n$};
\node (A10) at (4.05,0) {$0,$};
\node (A11) at (4.45,0) {$\bar1,$};
\node (A11) at (4.85,0) {$\bar1,$};
\node (A12) at (5.25,0) {$0,$};
\node (A13) at (5.65,0) {$1,$};
\node (A10) at (6.05,0) {$1,$};
\node (A10) at (6.45,0) {$0,$};
\node (A11) at (6.85,0) {$\bar1,$};
\node (A13) at (7.65,0) {$\bar1,\cdots\!,$};
\node (A14) at (8.55,0) {$0,$};
\node (A14) at (8.95,0) {$\bar1,$};
\node (A14) at (9.35,0) {$\bar1,$};
\node (A15) at (9.75,0) {$0;$};
\node (A16) at (10.15,0.03) {$1$};
\node (D3) at (3.4,-0.4) {${}$};
\draw (.7,-0.25) to (.7,-0.4) to (D3) node {$\cic{1}$} to (6.0,-0.4)  to (6.0,-0.25);
\node (C1) at (4.2,-0.6) {${}$};
\node (D4) at (3.95,0.6) {${}$};
\draw (1.5,-0.25) to (1.5,-0.6) to (C1) node {$\cic{3}$} to (6.8,-0.6)  to (6.8,-0.25);
\draw (1.1,0.25) to (1.1,0.6) to (D4) node {$\cic{2}$} to (6.4,0.6)  to (6.4,0.25);
\draw[dashed] (1.8,0.7) to (1.8,-0.7);
\draw[dashed] (3.8,0.7) to (3.8,-0.7);
\draw[dashed] (5.05,0.7) to (5.05,-0.7);
\draw[dashed] (6.25,0.7) to (6.25,-0.7);
\draw[dashed] (7.45,0.7) to (7.45,-0.7);
\draw[dashed] (8.3,0.7) to (8.3,-0.7);
\draw[dashed] (9.55,0.7) to (9.55,-0.7);
\node (lab) at (4,-1.2) {Possible cuts of $D_{6n+7} \zeta^\fm(\{1,\bar2\}_k)$};
\end{tikzpicture} }
\end{equation}

\begin{equation} \label{fig:r=6n+7-333}
\text{
\begin{tikzpicture}[scale=0.9]
\node (A4) at (-1.4,0) {$0,$};
\node (A5) at (-1,0) {$1,$};
\node (A6) at (-0.6,0) {$0,$};
\draw[dashed] (-0.3,0.7) to (-0.3,-0.7);
\node (A7) at (0.1,0) {$\cdots\!,$};
\draw[dashed] (0.5,0.7) to (0.5,-0.7);
\node (A7) at (0.75,0) {$0,$};
\node (A7) at (1.15,0) {$1,$};
\node (A9) at (1.55,0) {$0,$};
\node (A9) at (2.8,0.04)  {$\{010010\}_n$};
\node (A10) at (4.05,0) {$0,$};
\node (A11) at (4.45,0) {$1,$};
\node (A11) at (4.85,0) {$0,$};
\node (A12) at (5.25,0) {$0,$};
\node (A13) at (5.65,0) {$1,$};
\node (A10) at (6.05,0) {$0,$};
\node (A10) at (6.45,0) {$0,$};
\node (A11) at (6.85,0) {$1,$};
\node (A13) at (7.65,0) {$0,\cdots\!,$};
\node (A14) at (8.55,0) {$0,$};
\node (A14) at (8.95,0) {$1,$};
\node (A14) at (9.35,0) {$0,$};
\node (A15) at (9.75,0) {$0;$};
\node (A16) at (10.15,0.03) {$1$};
\node (C1) at (4.2,-0.6) {${}$};
\node (D4) at (3.95,0.6) {${}$};
\draw (1.5,-0.25) to (1.5,-0.6) to (C1) node {$\cic{2}$} to (6.8,-0.6)  to (6.8,-0.25);
\draw (1.1,0.25) to (1.1,0.6) to (D4) node {$\cic{1}$} to (6.4,0.6)  to (6.4,0.25);
\draw[dashed] (1.8,0.7) to (1.8,-0.7);
\draw[dashed] (3.8,0.7) to (3.8,-0.7);
\draw[dashed] (5.05,0.7) to (5.05,-0.7);
\draw[dashed] (6.25,0.7) to (6.25,-0.7);
\draw[dashed] (7.45,0.7) to (7.45,-0.7);
\draw[dashed] (8.3,0.7) to (8.3,-0.7);
\draw[dashed] (9.55,0.7) to (9.55,-0.7);
\node (lab) at (5,-1.2) {Possible cuts of $D_{6n+7} \zeta^\fm(3_k)$};
\end{tikzpicture} }
\end{equation}
It is clear that $\ncic{3}\,${} has an anti-symmetric subsequence so that it disappears and $\ncic{1}+\ncic{2}=0$ in both pictures.
This implies that \eqref{equ:D_r1bar2} holds for $r=6n+7$.

\medskip
By combining all the cases for different $r$'s above, we see that  \eqref{equ:D_r1bar2} holds
for all odd $r<k$. By Theorem~\ref{thm-Glanois} there is some $c\in\Q$
such that $8^k \zeta^\fm(\{1,\bar2\}_k)=\zeta^\fm(3_k)+c\,\zeta^\fm(3k)$.
Applying the period map \eqref{equ:periodMap} and using formula \eqref{equ:12bar}
we finally see that $c=0$ which implies \eqref{equ:1bar2Motivic} (or its restatement \eqref{equ:1bar2MotivicAgain}) immediately.

\subsection{Proof of base cases of Theorem \ref{thm:2bar3Motivic}}
We first prove Theorem \ref{thm:2bar3Motivic} when $a=0$ and $a=1$ as part of the base cases.
We will then prove the full Theorem \ref{thm:2bar3Motivic} by induction in the next section.

\begin{lem} \label{lem:2bar3Motivic}
For all integers $\ell\ge 0$ the motivic Euler sums
\begin{align}\label{equ:2bar3}
&\zeta^\fm(\{\bar2,3\}_\ell), \qquad \zeta^\fm(\{\bar2,3\}_\ell,\bar2), \\
&\zeta_1^\fm(\{\bar2,3\}_\ell), \qquad \zeta_1^\fm(\{\bar2,3\}_\ell,\bar2),\label{equ:1-2bar3} \\
&\zeta^\fm(\{3,\bar2\}_\ell), \qquad \zeta^\fm(\{3,\bar2\}_\ell,3), \label{equ:3bar2}\\
&\zeta_1^\fm(\{3,\bar2\}_\ell), \qquad \zeta_1^\fm(\{3,\bar2\}_\ell,3) \label{equ:13bar2}
\end{align}
are all unramified.
\end{lem}

We originally only planned to prove the pair of families in \eqref{equ:2bar3}
are both unramified.  When we tried the induction we found that we had to use a ``quadruple'' induction to
prove all the four pairs in Lemma~\ref{lem:2bar3Motivic} are unramified. This is what will do below.

By simple computation or using the datamine \cite{BlumleinBrVe2010} we get
\begin{align*}
    \zeta(\bar2,3)=&\, -\tfrac{1}{8}\zeta(3,2)-\tfrac{3}{32}\zeta(5),\\
    \zeta(\bar2,3,\bar2)=&\, \tfrac{7}{640}\zeta(7)-\tfrac{21}{640}\zeta(3,4),\\
 \zeta_1(\bar2,3)=&\, - \big(2\zeta(\bar3,3)+3 \zeta(\bar2,4) \big)=\tfrac{3}{8}\zeta(3)^2+\tfrac{5}{64}\zeta(6),\\
 \zeta_1(\bar2,3,\bar2)=&\,-\big(2\zeta(\bar3,3,\bar2)+3 \zeta(\bar2,4,\bar2)+2\zeta(\bar2,3,\bar3) \big)\\
   =&\, \tfrac{41}{32}\zeta(6,2)-\tfrac{185}{64}\zeta(1,7)-\frac{1}{8}\zeta(2,3,3)-\tfrac{215}{256}\zeta(8),\\
 \zeta(3,\bar2)=&\, \tfrac{7}{32}\zeta(5)-\tfrac{1}{4}\zeta(2,3),\\
\zeta(3,\bar2,3)=&\,\tfrac{3}{64}\zeta(6,2)-\tfrac{3}{16}\zeta(3,5)+\tfrac{3}{16}\zeta(3,3,2)-\tfrac{1}{48}\zeta(8),\\
\zeta_1(3,\bar2)=&\,- \big(2\zeta(3,\bar3)+3 \zeta(4,\bar2) \big)= \tfrac{147}{64}\zeta(6)-\tfrac{9}{8}\zeta(3)^2,\\
\zeta_1(3,\bar2,3)=&\,-\big(3\zeta(4,\bar2,3)+2\zeta(3,\bar3,3)+3\zeta(3,\bar2,4) \big)\\
 =&\,\tfrac{153}{64}\zeta(7,2)-\tfrac{21}{4}\zeta(6,3)-\tfrac{387}{64}\zeta(3,6)+\tfrac{11}{8}\zeta(9) .
\end{align*}
This convinced us that the base cases are all correct even though we need to do this on the motivic level.
First, we can see easily that $D_1=0$ for all the motivic Euler sums in Theorem~\ref{thm:2bar3Motivic} by Lemma \ref{lem:D1}.

\subsubsection{The base cases}
Lemma~\ref{lem:2bar3Motivic} clearly holds for $\ell=0$. To see the basic ideas we now prove $\ell=1$ case.
To save space, we often suppress the commas in the sub- and quotient sequence in the rest of the paper.
Using the left picture below we can check that
\begin{align*}
    D_3\zeta^\fm(\bar2,3)=&\,I^\fl(\bar1;010;0)\ot I^\fm(0;\bar10;1)+I^\fl(0;100;1)\ot I^\fm(0;\bar10;1)\\
        =&\,(\zeta^\fl(3)+\zeta_1^\fl(\bar2))\ot \zeta^\fm(\bar2)
        =(\zeta^\fl(3)-2\zeta^\fl(\bar3))\ot \zeta^\fm(\bar2) \in \calL_3\ot \calH_2,
\end{align*}
by (I3) and Prop.~\ref{prop-singleES}, which shows that $\zeta^\fm(\bar2,3)\in\calH_5$ by Theorem \ref{thm-Glanois}.
\begin{center}
\begin{tikzpicture}[scale=0.9]
\node (A0) at (0.05,0) {$0;$};
\node (A1) at (0.45,0) {$\bar1,$};
\node (A2) at (0.85,0) {$0,$};
\node (A3) at (1.25,0) {$1,$};
\node (B3) at (1.2,0.4) {${}$};
\node (A4) at (1.65,0) {$0,$};
\node (D4) at (1.6,-0.4) {${}$};
\node (A5) at (2.05,0) {$0;$};
\node (A6) at (2.45,0) {$1$};
\draw (0.4,0.25) to (0.4,0.4) to (B3) node {$\cic{1}$} to (2.0,0.4)  to (2.0,0.25);
\draw (0.8,-0.25) to (0.8,-0.4) to (D4) node {$\cic{2}$} to (2.4,-0.4)  to (2.4,-0.25);
\node (lab) at (1.2,-1) {$D_3\zeta^\fm(\bar2,3)$};
\end{tikzpicture}
\qquad
\begin{tikzpicture}[scale=0.9]
\node (RA0) at (5.05,0) {$0;$};
\node (RA1) at (5.45,0) {$1,$};
\node (RA4) at (5.85,0) {$0,$};
\node (RA3) at (6.25,0) {$\bar1,$};
\node (RA4) at (6.65,0) {$0,$};
\node (RA5) at (7.05,0) {$0,$};
\node (RA6) at (7.45,0) {$\bar1,$};
\node (RA7) at (7.85,0) {$0;$};
\node (RA8) at (8.25,0) {$1\phantom{,}$};
\node (R1) at (6.2,0.4) {${}$};
\node (R2) at (6.6,0.6) {${}$};
\node (R3) at (7.0,-0.4) {${}$};
\node (R4) at (7.4,-0.6) {${}$};
\draw (5.4,0.25) to (5.4,0.4) to (R1) node {$\cic{1}$} to (7.0,0.4)  to (7.0,0.25);
\draw (5.8,0.25) to (5.8,0.6) to (R2) node {$\cic{2}$} to (7.4,0.6)  to (7.4,0.25);
\draw (6.2,-0.25) to (6.2,-0.4) to (R3) node {$\cic{3}$} to (7.8,-0.4)  to (7.8,-0.25);
\draw (6.6,-0.25) to (6.6,-0.6) to (R4) node {$\cic{4}$} to (8.2,-0.6)  to (8.2,-0.25);
\node (lab) at (6.5,-1) {$D_3\zeta^\fm(\bar2,3,\bar2)$};
\end{tikzpicture}
\qquad
\begin{tikzpicture}[scale=0.9]
\node (RA0) at (5.05,0) {$0;$};
\node (RA1) at (5.45,0) {$1,$};
\node (RA2) at (5.85,0) {$0,$};
\node (RA3) at (6.25,0) {$\bar1,$};
\node (RC3) at (6.2,-0.4) {${}$};
\node (RA4) at (6.65,0) {$0,$};
\node (RB4) at (6.6,0.4) {${}$};
\node (RA5) at (7.05,0) {$0,$};
\node (RA6) at (7.45,0) {$\bar1,$};
\node (RA7) at (7.85,0) {$0;$};
\node (RA8) at (8.25,0) {$1\phantom{,}$};
\node (RD5) at (7.0,0.6) {${}$};
\draw (5,-0.25) to (5,-0.4) to (RC3) node {$\cic{0}$} to (7.4,-0.4)  to (7.4,-0.25);
\draw (5.4,0.25) to (5.4,0.4) to (RB4) node {$\cic{\bar1}$} to (7.8,0.4)  to (7.8,0.25);
\draw (5.8,0.25) to (5.8,0.6) to (RD5) node {$\cic{\bar2}$} to (8.2,0.6)  to (8.2,0.25);
\node (lab) at (6.5,-1) {$D_5\zeta^\fm(\bar2,3,\bar2)$};
\end{tikzpicture}
\end{center}

To compute $D_3\zeta^\fm(\bar2,3,\bar2)$ we can use the middle picture above to see that
$\ncic{1}+\ncic{4}=0$ and $\ncic{2}+\ncic{3}=0$ by path reversal. Hence, $D_3\zeta^\fm(\bar2,3,\bar2)=0$.
Further, using the right picture above we get
\begin{align*}
    D_5\zeta^\fm(\bar2,3,\bar2)=&\,I^\fl(0;10\bar100;\bar1)\ot I^\fm(0;\bar10;1)
    +\big(I^\fl(1;0\bar100\bar1;0) +I^\fl(0;\bar100\bar10;1) \big)\ot I^\fm(0;10;1)\\
 =&\,I^\fl(0;\bar10100;1)\ot I^\fm(0;\bar10;1)=\zeta^\fl(\bar2, 3)\ot \zeta^\fm(\bar2)
\end{align*}
by homothety and path reversal. Thus by the above proof that $\zeta^\fm(\bar2, 3)\in \calH_5$ and Prop.~\ref{prop-singleES}
we get $D_5\zeta^\fm(\bar2,3,\bar2)\in \calL_5\ot \calH_2.$

The above computation shows that both motivic Euler sums in \eqref{equ:2bar3} are unramified if $\ell=1$ by Lemma~\ref{lem:D1} and Theorem~\ref{thm-Glanois}.
One can similarly check that
$\zeta_1^\fm(\bar2,3), \zeta_1^\fm(\bar2,3,\bar2), \zeta^\fm(3,\bar2), \zeta^\fm(3,\bar2,3)$
are all unramified. Indeed, by similar computation as above, we get
\begin{align*}
D_3\zeta_1^\fm(\bar2,3)=&\,  I^\fl(0;0\bar10;1)\ot I^\fm(0;100;1)+
   \big( I^\fl(\bar1;010;0)+I^\fl(0;100;1)\big)\ot I^\fm(0;0\bar10;1)\\
=&\, \zeta_1^\fl(\bar2)\ot\zeta^\fm(3)+\big(\zeta^\fl(3)-\zeta_1^\fl(\bar2)\big)\ot \zeta_1^\fm(\bar2) \in \calL_3\ot \calH_3, \\
D_3\zeta_1^\fm(\bar2,3,\bar2)=&\,  I^\fl(0;010;\bar1)\ot I^\fm(0;\bar100\bar10;1)
   +\big( I^\fl(1;0\bar10;0)+I^\fl(0;\bar100;\bar1)\big)\ot I^\fm(0;010\bar10;1) \\
  &\,  + \big(I^\fl(\bar1;00\bar1;0) + I^\fl(0;0\bar10;1) \big)\ot I^\fm(0;010\bar10;1)  \\
=&\, \zeta_1^\fl(\bar2)\ot\zeta^\fm(3,\bar2)  \in \calL_3\ot \calH_5, \\
D_5\zeta_1^\fm(\bar2,3,\bar2)=&\,I^\fl(0;10\bar100;\bar1)\ot I^\fm(0;0\bar10;1)
    +\big(I^\fl(1;0\bar100\bar1;0) +I^\fl(0;\bar100\bar10;1) \big)\ot I^\fm(0;010;1)\\
 =&\, \zeta^\fl(\bar2, 3)\ot \zeta_1^\fm(\bar2)  \in \calL_5\ot \calH_3
\end{align*}
by essentially the same proof of unramification of $\zeta^\fm(\bar2,3)$ above.
\begin{center}
\begin{tikzpicture}[scale=0.9]
\node (A1) at (0.45,0) {$0;$};
\node (A2) at (0.85,0) {$0,$};
\node (A3) at (1.25,0) {$\bar1,$};
\node (A4) at (1.65,0) {$0,$};
\node (D4) at (1.6,0.4) {${}$};
\node (A5) at (2.05,0) {$0,$};
\node (C5) at (2.0,-0.4) {${}$};
\node (A6) at (2.45,0) {$\bar1,$};
\node (C6) at (2.4,-0.6) {${}$};
\node (A7) at (2.85,0) {$0;$};
\node (A8) at (3.25,0) {$1$};

\draw (0.8,0.25) to (0.8,0.4) to (D4) node {$\cic{1}$} to (2.4,0.4)  to (2.4,0.25);
\draw (1.2,-0.25) to (1.2,-0.4) to (C5) node {$\cic{2}$} to (2.8,-0.4)  to (2.8,-0.325);
\draw (1.6,-0.25) to (1.6,-0.6) to (C6) node {$\cic{3}$} to (3.2,-0.6)  to (3.2,-0.25);
\node (lab) at (1.8,-1.2) {$D_3\zeta_1^\fm(3,\bar2)$};
\end{tikzpicture}
\qquad
\begin{tikzpicture}[scale=0.9]
\node (A1) at (0.45,0) {$0;$};
\node (A2) at (0.85,0) {$0,$};
\node (A3) at (1.25,0) {$\bar1,$};
\node (A4) at (1.65,0) {$0,$};
\node (D4) at (1.6,0.4) {${}$};
\node (A5) at (2.05,0) {$0,$};
\node (C5) at (2.0,-0.4) {${}$};
\node (A6) at (2.45,0) {$\bar1,$};
\node (C6) at (2.4,-0.6) {${}$};
\node (A7) at (2.85,0) {$0,$};
\node (A8) at (3.25,0) {$1,$};
\node (D8) at (3.2,0.6) {${}$};
\node (A9) at (3.65,0) {$0,$};
\node (D9) at (3.6,0.4) {${}$};
\node (A10) at (4.05,0) {$0;$};
\node (A11) at (4.45,0) {$1,$};

\draw (0.8,0.25) to (0.8,0.4) to (D4) node {$\cic{1}$} to (2.4,0.4)  to (2.4,0.25);
\draw (1.2,-0.25) to (1.2,-0.4) to (C5) node {$\cic{2}$} to (2.8,-0.4)  to (2.8,-0.325);
\draw (1.6,-0.25) to (1.6,-0.6) to (C6) node {$\cic{3}$} to (3.2,-0.6)  to (3.2,-0.25);
\draw (2.4,0.25) to (2.4,0.6) to (D8) node {$\cic{4}$} to (4,0.6)  to (4,0.25);
\draw (2.8,0.25) to (2.8,0.4) to (D9) node {$\cic{5}$} to (4.4,0.4)  to (4.4,0.25);
\node (lab) at (2.7,-1.2) {$D_3\zeta_1^\fm(3,\bar2,3)$};
\end{tikzpicture}
\qquad
\begin{tikzpicture}[scale=0.9]
\node (A1) at (0.45,0) {$0;$};
\node (A2) at (0.85,0) {$0,$};
\node (A3) at (1.25,0) {$\bar1,$};
\node (A4) at (1.65,0) {$0,$};
\node (D4) at (2,0.4) {${}$};
\node (A5) at (2.05,0) {$0,$};
\node (C5) at (2.4,-0.4) {${}$};
\node (A6) at (2.45,0) {$\bar1,$};
\node (C6) at (3.2,-0.6) {${}$};
\node (A7) at (2.85,0) {$0,$};
\node (A8) at (3.25,0) {$1,$};
\node (A9) at (3.65,0) {$0,$};
\node (A10) at (4.05,0) {$0;$};
\node (A11) at (4.45,0) {$1,$};

\draw (0.8,0.25) to (0.8,0.4) to (D4) node {$\cic{3}$} to (3.2,0.4)  to (3.2,0.25);
\draw (1.2,-0.25) to (1.2,-0.4) to (C5) node {$\cic{2}$} to (3.6,-0.4)  to (3.6,-0.325);
\draw (2,-0.25) to (2,-0.6) to (C6) node {$\cic{1}$} to (4.4,-0.6)  to (4.4,-0.25);
\node (lab) at (2.5,-1.2) {$D_5\zeta_1^\fm(3,\bar2,3)$};
\end{tikzpicture}
\end{center}
Moreover, using  the above pictures ($\ncic{1}\,${} and $\ncic{2}\,${} always cancel) one checks easily that
\begin{align*}
D_3\zeta_1^\fm(3,\bar2)=&\,
    I^\fl(0;0\bar10;1)\ot I^\fm(0;0\bar10;1) =\zeta_1^\fl(\bar2)\ot \zeta_1^\fm(\bar2)  \in \calL_3\ot \calH_3, \\
D_5\zeta_1^\fm(3,\bar2)=&\, 0, \\
D_3\zeta_1^\fm(3,\bar2,3)=&\,
I^\fl(0;0\bar10;1)\ot I^\fm(0;0\bar10100;1) + \big( I^\fl(\bar1;010;0)+I^\fl(0;100;1)\big)\ot I^\fm(0;0\bar100\bar10;1) \\
=&\, \zeta_1^\fl(\bar2)\ot \zeta_1^\fm(\bar2,3)+\big(\zeta^\fl(3)- \zeta_1^\fl(\bar2)\big)\ot \zeta_1^\fm(3,\bar2) \in \calL_3\ot \calH_6, \\
D_5\zeta_1^\fm(3,\bar2,3)=&\, I^\fl(0;\bar100\bar10;1)\ot I^\fm(0;0100;1)
    =\zeta^\fl(3,\bar2)\ot \zeta_1^\fm(3) \in \calL_5\ot \calH_4, \\
D_7\zeta_1^\fm(3,\bar2,3)=&0,
\end{align*}
using the known unramified motivic Euler sums. Similarly,
\begin{align*}
D_3\zeta^\fm(3,\bar2)=&\,
    I^\fl(0;0\bar10;1)\ot I^\fm(0;\bar10;1) =\zeta_1^\fl(\bar2)\ot \zeta^\fm(\bar2)  \in \calL_3\ot \calH_2, \\
D_3\zeta^\fm(3,\bar2,3)=&\,
I^\fl(0;0\bar10;1)\ot I^\fm(0;\bar10100;1) + \big( I^\fl(\bar1;010;0)+I^\fl(0;100;1)\big)\ot I^\fm(0;\bar100\bar10;1) \\
=&\,\zeta_1^\fl(\bar2)\ot \zeta^\fm(\bar2,3)+\big(\zeta^\fl(3)- \zeta_1^\fl(\bar2)\big)\ot \zeta^\fm(3,\bar2) \in \calL_3\ot \calH_5, \\
D_5\zeta^\fm(3,\bar2,3)=&\, I^\fl(0;\bar100\bar10;1)\ot I^\fm(0;100;1)
    =\zeta^\fl(3,\bar2)\ot \zeta^\fm(3)\in \calL_5\ot \calH_3.
\end{align*}
The above computation shows that all the motivic Euler sums in Lemma~\ref{lem:2bar3Motivic} are unrmified when $\ell=1$.

We now assume that $k\ge 2$ and that the motivic Euler sums in \eqref{equ:2bar3} are unramified for all $\ell<k$.

\subsection{The inductive step for $\zeta^\fm(\{\bar2,3\}_k)$}
As the sequence $\rho(\{\bar2,3\}_k)$ has a period of 10, it is clear that $D_9=0$
and the pattern of $D_{10+r}$ is basically the same as that of $D_r$ for any $r>1$, with
one more of $\{\bar2,3\}_2$ in the subsequence and one less of $\{\bar2,3\}_2$ in the quotient.
Thus it suffices for us to show that $D_r\zeta^\fm(\{\bar2,3\}_k)$ are stable derivations for $r=10n+3,10n+5,10n+7,10n+11$,
meaning that all its factors are of similar forms which are then unramified by induction.

\subsubsection{$r=10n+3$}\label{subsec:2bar3caser=3}
We have the following picture when $k$ is odd and we only need to modify the left end by exchange 1 and $\bar1$ when $k$ is even.

\begin{center}
\begin{tikzpicture}[scale=0.9]
\node (A0) at (-2.75,0) {$0;$};
\node (A1) at (-2.35,0) {$\bar1,$};
\node (A2) at (-1.95,0) {$0,$};
\node (A3) at (-1.55,0) {$1,$};
\node (A13) at (-0.75,0) {$0,\ldots,$};
\node (A0) at (0.05,0) {$0,$};
\node (A1) at (0.45,0) {$1,$};
\node (A2) at (0.85,0) {$0,$};
\node (A3) at (1.25,0) {$\bar1,$};
\node (A4) at (1.65,0) {$0,$};
\node (A5) at (2.85,0) {$\{S_{\bar1},S_{1}\}_n$};
\node (A10) at (4.05,0) {$0,$};
\node (A11) at (4.45,0) {$\bar1,$};
\node (A11) at (4.85,0) {$0,$};
\node (A12) at (5.25,0) {$1,$};
\node (A13) at (6.05,0) {$0,\ldots,$};
\node (A14) at (6.95,0) {$0,$};
\node (A14) at (7.35,0) {$\bar1,$};
\node (A14) at (7.75,0) {$0,$};
\node (A15) at (8.15,0) {$1,$};
\node (A16) at (8.55,0) {$0,$};
\node (A17) at (8.95,0) {$0;$};
\node (A17) at (9.35,0) {$1\phantom{,}$};

\node (B3) at (2.2,0.4) {${}$};
\node (D4) at (2.6,-0.4) {${}$};
\node (C5) at (3.0,0.6) {${}$};
\node (C6) at (3.4,-0.6) {${}$};
\draw (0.4,0.25) to (0.4,0.4) to (B3) node {$\cic{1}$} to (4.0,0.4)  to (4.0,0.25);
\draw (0.8,-0.25) to (0.8,-0.4) to (D4) node {$\cic{2}$} to (4.4,-0.4)  to (4.4,-0.25);
\draw (1.2,0.25) to (1.2,0.6) to (C5) node {$\cic{3}$} to (4.8,0.6)  to (4.8,0.25);
\draw (1.6,-0.25) to (1.6,-0.6) to (C6) node {$\cic{4}$} to (5.2,-0.6)  to (5.2,-0.25);
\draw[dashed] (-.95,0.7) to (-.95,-.7);
\draw[dashed] (-0.15,0.7) to (-0.15,-.7);
\draw[dashed] (1.85,0.7) to (1.85,-0.7);
\draw[dashed] (3.8,0.7) to (3.8,-.7);
\draw[dashed] (5.85,0.7) to (5.85,-.7);
\draw[dashed] (6.7,0.7) to (6.7,-.7);
\draw[dashed] (8.75,0.7) to (8.75,-.7);
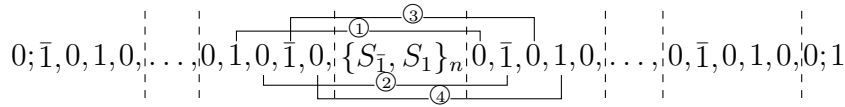
\node (lab) at (3.5,-1.2) {Possible cuts of $D_{10n+3}\zeta^\fm(\{\bar2,3\}_k)$};
\end{tikzpicture}
\end{center}
We call a subsequence $S_{\bar1}=(0\bar1010)$ a $S_{\bar1}$-block, and a subsequence $S_{1}=(010\bar10)$ a $S_{1}$-block.
For general $k$, the sequence $\rho(\{\bar2,3\}_k)$ always starts with a $S_{(-1)^k}$-block (here we identify $-1$ with $\bar 1$) and ends with a $S_{\bar1}$-block followed by $0,1$. For example, the picture above starts with a $S_{\bar1}$-block
$I^\fm(S_{\bar1},\ldots)=I^\fm(0;\bar1,0,1,0,\ldots).$

If a subsequence starts in a $S_1$-block and ends in a $S_{\bar1}$-block (as shown in the picture above)
then there are only four possible nonzero cuts with the same quotient sequence:
$$\aligned
{\ncic{1}}=&\, I^\fl(1;0,\bar1,0,\{S_{\bar1},S_{1}\}_n;0)\ot I^\fm(0;\bar1,\dots,0;1), \\
{\ncic{2}}=&\, I^\fl(0;\bar1,0,\{S_{\bar1},S_{1}\}_n,0;\bar1)\ot I^\fm(0;\bar1,\dots,0;1), \\
{\ncic{3}}=&\, I^\fl(\bar1;0,\{S_{\bar1},S_{1}\}_n,0,\bar1;0)\ot I^\fm(0;\bar1,\dots,0;1), \\
{\ncic{4}}=&\, I^\fl(0;\{S_{\bar1},S_{1}\}_n,0,\bar1,0;1)\ot I^\fm(0;\bar1,\dots,0;1).
\endaligned$$
By path reversal and homothety, we see immediately that $\ncic{1}+\ncic{4}=\ncic{2}+\ncic{3}=0$.

If a subsequence starts in a $S_{\bar1}$-block and ends in a $S_1$-block, then the argument is completely similar.
We only need to exchange $1$ and $\bar1$ and therefore the four terms still cancel each other.

Hence, we only need to consider the subsequences ending with the last 0 or the last 1.
Both of them must start in some $S_{\bar1}$-block and we thus get two terms
$$\aligned
{\ncic{\bar1}}=&\, I^\fl(\bar1;0,1,0,\{S_{1},S_{\bar1}\}_n;0)\ot I^\fm(\{S_{\bar1},S_1\}_{(k-1)/2-n},0,\bar1,0;1),\\
{\ncic{\bar2}}=&\, I^\fl(0;1,0,\{S_{1},S_{\bar1}\}_n,0;1)\ot I^\fm(\{S_{\bar1},S_1\}_{(k-1)/2-n},0,\bar1,0;1)
\endaligned$$
when $k$ is odd and
$$\aligned
{\ncic{1}}=&\, I^\fl(\bar1;0,1,0,\{S_{1},S_{\bar1}\}_n;0)\ot I^\fm(S_1,\{S_{\bar1},S_1\}_{k/2-n-1},0,\bar1,0;1), \\
{\ncic{2}}=&\, I^\fl(0;1,0,\{S_{1},S_{\bar1}\}_n,0;1)\ot I^\fm(S_1,\{S_{\bar1},S_1\}_{k/2-n-1},0,\bar1,0;1)
\endaligned$$
when $k$ is even. In both cases we can rewrite the two terms as (by using path reversal for the first)
$$\aligned
{\ncic{1}}=&\, -\zeta_1^\fl(\{\bar2,3\}_{2n},\bar2)\ot \zeta^\fm(\{\bar2,3\}_{k-2n-1},\bar2)\in\calL_{10n+3}\ot \calH_{5k-10n-3}, \\
{\ncic{2}}=&\,\zeta^\fl(\{3,\bar2\}_{2n},3)\ot \zeta^\fm(\{\bar2,3\}_{k-2n-1},\bar2)\in\calL_{10n+3}\ot \calH_{5k-10n-3}
\endaligned$$
since all the factors are unramified by induction as $k>2n>0$. For examples, when $k=5$ and $n=1$, we have
$$\aligned
{\ncic{\bar1}}=&\, I^\fl(\bar1;010010\bar100\bar1010;0)\ot I^\fm(0;\bar1010010\bar100\bar10;1),\\
=&\, -I^\fl(0;0\bar1010010\bar100\bar10;1)\ot I^\fm(0;\bar1010010\bar100\bar10;1)
=-\zeta_1^\fl(\bar2,3,\bar2,3,\bar2)\ot \zeta^\fm(\bar2,3,\bar2,3,\bar2),
\endaligned$$
and when $k=4$ and $n=1$, we have
$$
{\ncic{2}}= I^\fl(0;10010\bar100\bar10100;1)\ot I^\fm(0;10\bar100\bar10;1)
=\zeta^\fl(3,\bar2,3,\bar2,3)\ot \zeta^\fm(\bar2,3,\bar2).
$$

\subsubsection{$r=10n+5$}\label{subsec:2bar3caser=5}
We have the following picture when $k$ is odd and we only need to modify the left end by exchange 1 and $\bar1$ when $k$ is even.

\begin{center}
\begin{tikzpicture}[scale=0.9]
\node (A0) at (-2.75,0) {$0;$};
\node (A1) at (-2.35,0) {$\bar1,$};
\node (A2) at (-1.95,0) {$0,$};
\node (A3) at (-1.55,0) {$1,$};
\node (A13) at (-0.75,0) {$0,\ldots,$};
\node (A0) at (0.05,0) {$0,$};
\node (A1) at (0.45,0) {$1,$};
\node (A2) at (0.85,0) {$0,$};
\node (A3) at (1.25,0) {$\bar1,$};
\node (A4) at (1.65,0) {$0,$};
\node (A5) at (2.85,0) {$\{S_{\bar1},S_{1}\}_n$};
\node (A10) at (4.05,0) {$0,$};
\node (A11) at (4.45,0) {$\bar1,$};
\node (A11) at (4.85,0) {$0,$};
\node (A12) at (5.25,0) {$1,$};
\node (A13) at (6.05,0) {$0,\ldots,$};
\node (A14) at (6.95,0) {$0,$};
\node (A14) at (7.35,0) {$\bar1,$};
\node (A14) at (7.75,0) {$0,$};
\node (A15) at (8.15,0) {$1,$};
\node (A16) at (8.55,0) {$0,$};
\node (A17) at (8.95,0) {$0;$};
\node (A17) at (9.35,0) {$1\phantom{,}$};

\node (B3) at (2.6,0.4) {${}$};
\node (D4) at (3,-0.6) {${}$};
\node (C5) at (3.4,0.6) {${}$};
\node (C6) at (2.2,-0.4) {${}$};
\draw (0.4,0.25) to (0.4,0.4) to (B3) node {$\cic{1}$} to (4.8,0.4)  to (4.8,0.25);
\draw (0.8,-0.25) to (0.8,-0.6) to (D4) node {$\cic{2}$} to (5.2,-0.6)  to (5.2,-0.25);
\draw (1.2,0.25) to (1.2,0.6) to (C5) node {$\cic{3}$} to (5.6,0.6)  to (5.6,0.325);
\draw (0,-0.25) to (0,-0.4) to (C6) node {$\cic{0}$} to (4.4,-0.4)  to (4.4,-0.25);
\draw[dashed] (-.95,0.7) to (-.95,-.7);
\draw[dashed] (-0.15,0.7) to (-0.15,-.7);
\draw[dashed] (1.85,0.7) to (1.85,-0.7);
\draw[dashed] (3.8,0.7) to (3.8,-.7);
\draw[dashed] (5.85,0.7) to (5.85,-.7);
\draw[dashed] (6.7,0.7) to (6.7,-.7);
\draw[dashed] (8.75,0.7) to (8.75,-.7);
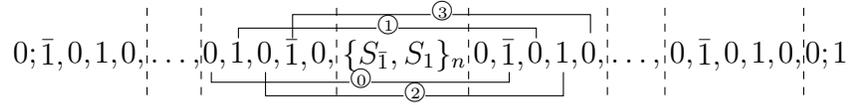
\node (lab) at (3.5,-1.2) {Possible cuts of $D_{10n+5}\zeta^\fm(\{\bar2,3\}_k)$};
\end{tikzpicture}
\end{center}

If a subsequence starts in a $S_{1}$-block as shown above
then there are only four possible nonzero terms with the same quotient sequence in $\ncic{1}\,${} and $\ncic{2}$:
$$\aligned
{\ncic{0}}=&\, I^\fl(0;1,0,\bar1,0,\{S_{\bar1},S_{1}\}_n,0;\bar1)\ot I^\fm(0;\dots,S_{\bar1},0\bar1010,S_1,\dots;1), \\
{\ncic{1}}=&\, I^\fl(1;0,\bar1,0,\{S_{\bar1},S_{1}\}_n,0,\bar1;0)\ot I^\fm(0;\dots,S_{\bar1},01010,S_1,\dots;1), \\
{\ncic{2}}=&\, I^\fl(0;\bar1,0,\{S_{\bar1},S_{1}\}_n,0,\bar1,0;1)\ot I^\fm(0;\dots,S_{\bar1},01010,S_1,\dots;1), \\
{\ncic{3}}=&\, I^\fl(\bar1;0,\{S_{\bar1},S_{1}\}_n,0,\bar1,0,1;0)\ot I^\fm(0;\dots,S_{\bar1},010\bar10,S_1,\dots;1).
\endaligned$$
By path reversal and homothety, we see easily that $\ncic{1}+\ncic{2}=0$. But we don't see
the immediate cancelation of $\ncic{0}\,${} and $\ncic{3}\,${}. However, all possible cuts of $\ncic{0}\,${} and $\ncic{\bar0}\,${} together yield
\begin{equation*}
\zeta^\fl(\{\bar2,3\}_{2n+1})\ot \sum_{j=0}^{k-2n-1}\zeta^\fm(\{2,\bar3\}_{k-2n-1-j},\{\bar2,3\}_{j})
\end{equation*}
while all possible cuts of $\ncic{3}\,${} and $\ncic{\bar3}\,${} together yield
\begin{equation*}
-\zeta^\fl(\{\bar2,3\}_{2n+1})\ot \sum_{j=0}^{k-2n-2}\zeta^\fm(\{2,\bar3\}_{k-2n-1-j},\{\bar2,3\}_{j})
\end{equation*}
by path reversal. Note that there is one extra term for $\ncic{0}\,${}
when a subsequence ends at the last 1. Therefore
$$
D_5\zeta^\fm(\{\bar2,3\}_k)=\zeta^\fl(\{\bar2,3\}_{2n+1})\ot\zeta^\fm(\{\bar2,3\}_{k-2n-1}) \in \calL_{10n+5}\ot  \calH_{5k-10n-5}
$$
by induction.

\subsubsection{$r=10n+7$}\label{subsec:2bar3caser=7}
We have the following picture when $k$ is odd.

\begin{center}
\begin{tikzpicture}[scale=0.9]
\node (A0) at (-2.75,0) {$0;$};
\node (A1) at (-2.35,0) {$\bar1,$};
\node (A2) at (-1.95,0) {$0,$};
\node (A3) at (-1.55,0) {$1,$};
\node (A13) at (-0.75,0) {$0,\ldots,$};
\node (A0) at (0.05,0) {$0,$};
\node (A1) at (0.45,0) {$1,$};
\node (A2) at (0.85,0) {$0,$};
\node (A3) at (1.25,0) {$\bar1,$};
\node (A4) at (1.65,0) {$0,$};
\node (A5) at (2.85,0) {$\{S_{\bar1},S_{1}\}_n$};
\node (A10) at (4.05,0) {$0,$};
\node (A11) at (4.45,0) {$\bar1,$};
\node (A11) at (4.85,0) {$0,$};
\node (A12) at (5.25,0) {$1,$};
\node (A10) at (5.65,0) {$0,$};
\node (A10) at (6.05,0) {$0,$};
\node (A11) at (6.45,0) {$1,$};
\node (A11) at (6.85,0) {$0,$};
\node (A12) at (7.25,0) {$\bar1,$};
\node (A13) at (8.05,0) {$0,\ldots,$};
\node (A14) at (8.95,0) {$0,$};
\node (A14) at (9.35,0) {$\bar1,$};
\node (A14) at (9.75,0) {$0,$};
\node (A15) at (10.15,0) {$1,$};
\node (A16) at (10.55,0) {$0,$};
\node (A17) at (10.95,0) {$0;$};
\node (A17) at (11.35,0) {$1\phantom{,}$};

\node (B3) at (3,0.4) {${}$};
\node (C5) at (3.8,-0.6) {${}$};
\node (C6) at (2.6,-0.4) {${}$};
\draw (0.4,0.25) to (0.4,0.4) to (B3) node {$\cic{2}$} to (5.6,0.4)  to (5.6,0.25);
\draw (1.2,-0.25) to (1.2,-0.6) to (C5) node {$\cic{3}$} to (6.4,-0.6)  to (6.4,-0.25);
\draw (0,-0.25) to (0,-0.4) to (C6) node {$\cic{1}$} to (5.2,-0.4)  to (5.2,-0.25);
\draw[dashed] (-.95,0.7) to (-.95,-.7);
\draw[dashed] (-0.15,0.7) to (-0.15,-.7);
\draw[dashed] (1.85,0.7) to (1.85,-0.7);
\draw[dashed] (3.8,0.7) to (3.8,-.7);
\draw[dashed] (5.85,0.7) to (5.85,-.7);
\draw[dashed] (7.8,0.7) to (7.8,-.7);
\draw[dashed] (8.75,0.7) to (8.75,-.7);
\draw[dashed] (10.75,0.7) to (10.75,-.7);
\node (lab) at (3.5,-1.2) {Possible cuts of $D_{10n+7}\zeta^\fm(\{\bar2,3\}_k)$};
\end{tikzpicture}
\end{center}

If a subsequence starts in a $S_1$-block
then there are only three possible nonzero terms:
$$\aligned
{\ncic{1}}=&\, I^\fl(0;10\bar10,\{S_{\bar1},S_1\}_n,0\bar10;1)\ot I^\fm(0;\dots,S_{\bar1},0,1,0,S_1,\dots;1), \\
{\ncic{2}}=&\, I^\fl(1;0\bar10,\{S_{\bar1},S_1\}_n,0\bar101;0)\ot I^\fm(0;\dots,S_{\bar1},0,1,0,S_1,\dots;1), \\
{\ncic{3}}=&\, I^\fl(\bar1;0,\{S_{\bar1},S_1\}_n,0\bar10100;1)\ot I^\fm(0;\cdots;1).
\endaligned$$
Then  $\ncic{3}=0$ since its subsequence is anti-symmetric. The quotient sequence in $\ncic{1}\,${} and $\ncic{2}\,${} are the same.
By path reversal, we see immediately that $\ncic{1}+\ncic{2}=0$.

If a subsequence starts in a $S_{\bar1}$-block, then the argument is completely similar
so that all the terms will be canceled.

In conclusion we see that  $D_{10n+7}\zeta^\fm(\{\bar2,3\}_k)=0$.

\subsubsection{$r=10n+11$}\label{subsec:2bar3caser=11}
We notice that any such cut would produce the same type of quotient sequence with $n+1$ blocks of $(S_1,S_{\bar1})$ missing.
So the right factor is unramified by induction. The left factor must have one of the following forms:
\begin{align*}
I^\fl(0; \{S_1,S_{\bar1}\}_{n+1},0;1)=I^\fl(0;\{S_{\bar1},S_1\}_{n+1},0;\bar1)= &\, \zeta_1^\fl(\{\bar2,3\}_{2n+2})\in\calL_{10n+11},\\
I^\fl(1;0,\{S_1,S_{\bar1}\}_{n+1};0)=I^\fl(\bar1;0,\{S_{\bar1},S_1\}_{n+1};0)= &\, -\zeta_1^\fl(\{\bar2,3\}_{2n+2})\in\calL_{10n+11}.
\end{align*}
by inductive assumption.

\medskip
Combining all the findings in \S\ref{subsec:2bar3caser=3}-\S\ref{subsec:2bar3caser=11} we see that
$D_1\zeta^\fm(\{\bar2,3\}_k)=0$ and for all odd $r<5k$,
$D_r\zeta^\fm(\{\bar2,3\}_k)\in\calL_r\ot \calH_{5k-r}$, i.e., $D_r$ is a stable derivation. Consequently,
$\zeta^\fm(\{\bar2,3\}_k)$ is unramified by Lemma~\ref{lem:D1} and Theorem~\ref{thm-Glanois}.

\subsection{The inductive step for $\zeta^\fm(\{\bar2,3\}_k,\bar2)$}\label{subsec:bar23bar2}
The computation of $D_r\zeta^\fm(\{\bar2,3\}_\ell,\bar2)$ can be carried out in a similar manner as that for $\zeta^\fm(\{\bar2,3\}_\ell)$. We can attach $0;\bar1$ to the right end of the picture in \S\ref{subsec:2bar3caser=3} and then exchange $\bar1$ and 1. By the same argument there, $D_{10n+r}$ are all stable for $r=3,5,7,9,11$. For brevity, we only show the computation when $r=7$ in which case the induction assumption for motivic Euler sum $\zeta^\fm(\{\bar2,3\}_\ell)$ is essentially used. In this case,
all the computation will be the same as that of $D_7\zeta^\fm(\{\bar2,3\}_\ell)$ except for the
only one extra possible cut starting at the
beginning of the last $S_{\bar1}$-block for which we have the following picture.

\begin{center}
\begin{tikzpicture}[scale=0.9]
\node (R1) at (4.35,0) {$\dots,$};
\node (RA0) at (5.05,0) {$0,$};
\node (RA1) at (5.45,0) {$1,$};
\node (RA2) at (5.85,0) {$0,$};
\node (RA3) at (6.25,0) {$\bar1,$};
\node (RA4) at (6.65,0) {$0,$};
\node (RA5) at (7.05,0) {$0,$};
\node (RD5) at (6.6,0.8) {${}$};
\node (RA6) at (7.45,0) {$\bar1,$};
\node (RA7) at (7.85,0) {$0;$};
\node (RA8) at (8.25,0) {$1\phantom{,}$};
\draw (5,0.25) to (5,0.8) to (RD5) node {$\cic{\bar1}$} to (8.2,0.8)  to (8.2,0.25);
\draw[dashed] (4.8,0.7) to (4.8,-0.7);
\draw[dashed] (6.8,0.7) to (6.8,-0.7);
\end{tikzpicture}
\end{center}
which produces
\begin{equation*}
\ncic{\text{$\bar1$}}=I^\fl(0;10\bar100\bar10;1)\ot \zeta^\fm(\{\bar2,3\}_{k-1})
=\zeta^\fl(\bar2,3,\bar2)\ot \zeta^\fm(\{\bar2,3\}_{k-1})\in\calL_7\ot \calH_{5k-5}
\end{equation*}
by the initial step and the inductive assumption.

By Theorem~\ref{thm-Glanois} we have now proved that both
families of motivic Euler sums in \eqref{equ:2bar3} are unramified when $\ell=k$.

\subsection{The inductive step for $\zeta^\fm(\{3,\bar2\}_k)$ and $\zeta^\fm(\{3,\bar2\}_k,3)$}
We now prove that the pair of families in \eqref{equ:3bar2}
are both unramified. It is mostly the same as that of  \eqref{equ:2bar3} except that in the
inductive step we need to use the validity of unramification of \eqref{equ:2bar3}.

When $r=3$ (or similarly for $r=10n+3<5k$) we can modify the picture in \ref{subsec:2bar3caser=3}
at the two ends if $k$ is odd. We just add $0;\bar10$ to the left end and remove $0,0;1$ from the right end
(no right-end change needed for $\zeta^\fm(\{3,\bar2\}_k,3)$).
So we consider the following picture for even $k>1$:

\begin{center}
\begin{tikzpicture}[scale=0.9]
\node (A2) at (0.85,0) {$0;$};
\node (A3) at (1.25,0) {$\bar1,$};
\node (A4) at (1.65,0) {$0,$};
\node (A5) at (2.05,0) {$0,$};
\node (A6) at (2.45,0) {$\bar1,$};
\node (A7) at (2.85,0) {$0,$};
\node (A8) at (3.25,0) {$1,$};
\node (A9) at (3.65,0) {$0,$};
\node (A10) at (4.05,0) {$0,$};
\node (A11) at (4.45,0) {$1,$};
\node (A11) at (4.85,0) {$0,$};
\node (A12) at (5.25,0) {$\bar1,$};
\node (A13) at (6.05,0) {$0,\ldots,$};
\node (A14) at (6.95,-0.5) {$0,$};
\node (A14) at (7.35,-0.5) {$\bar1,$};
\node (A14) at (7.75,-0.5) {$0,$};
\node (A15) at (8.15,-0.5) {$1,$};
\node (A16) at (8.55,-0.5) {$0,$};
\node (A17) at (8.95,-0.5) {$0;$};
\node (A17) at (11.5,-0.49) {$1\phantom{,}$ \qquad {for $\zeta^\fm(\{3,\bar2\}_k,3)$}};
\node (A14) at (6.95,0.5) {$0,$};
\node (A14) at (7.35,0.5) {$\bar1,$};
\node (A14) at (7.75,0.5) {$0;$};
\node (A17) at (10.7,0.52) {$1\phantom{,}$  \qquad\qquad\ \ {for $\zeta^\fm(\{3,\bar2\}_k)$}};
\node (D4) at (1.6,0.6) {${}$};
\node (C5) at (2.0,-0.4) {${}$};
\node (C6) at (2.4,-0.6) {${}$};
\node (C17) at (8.5,-0.1) {${}$};
\draw (7.7,-0.25) to (7.7,-0.1) to (C17) node {$\cic{2}$} to (9.2,-0.1)  to (9.2,-0.25);
\draw (0.8,0.25) to (0.8,0.6) to (D4) node {$\cic{\bar2}$} to (2.4,0.6)  to (2.4,0.25);
\draw (1.2,-0.25) to (1.2,-0.4) to (C5) node {$\cic{\bar3}$} to (2.8,-0.4)  to (2.8,-0.325);
\draw (1.6,-0.25) to (1.6,-0.6) to (C6) node {$\cic{\bar4}$} to (3.2,-0.6)  to (3.2,-0.25);
\draw[dashed] (1.8,0.7) to (1.8,-0.7);
\draw[dashed] (3.8,0.7) to (3.8,-0.7);
\draw[dashed] (5.8,0.7) to (5.8,-0.7);
\draw[dashed] (6.7,0.7) to (6.7,-0.7);
\draw[dashed] (8.7,0.1) to (8.7,-0.7);
\end{tikzpicture}
\end{center}
Similar to what we have shown previously we see easily that  $\ncic{\bar2}+\ncic{\bar3}=0$. But $\ncic{\bar4}\,${}
from the first block cannot be canceled by the missing $\ncic{\bar1}${} in front of it. All the other cuts in $D_3\zeta^\fm(\{3,\bar2\}_k)$
will either cancel each other or vanish themselves (see \S\ref{subsec:2bar3caser=3}).
For general $D_{10n+3}$ the argument is the same and therefore
$$
D_{10n+3}\zeta^\fm(\{3,\bar2\}_k)
    =\zeta_1^\fl(\{\bar2,3\}_{2n},\bar2)\ot \zeta^\fm(\{\bar2,3\}_{k-2n-1},\bar2)
    \in \calL_{10n+3}\ot \calH_{5k-10n-3}
$$
by inductive assumption. Similarly, we see that
$$
D_{10n+3}\zeta^\fm(\{3,\bar2\}_k,3)=\zeta_1^\fl(\{\bar2,3\}_{2n},\bar2)\ot \zeta^\fm(\{\bar2,3\}_{k-2n})
    +\zeta^\fl(\{3,\bar2\}_{2n},3)\ot \zeta^\fm(\{3,\bar2\}_{k-2n})
$$
also lies in $\calL_{10n+3}\ot \calH_{5k-10n}$ by induction,
where the second term comes from the cut  $\ncic{2}\,${} ending at the last 1.

Applying the same idea we find that all of  $D_{10n+5}$, $D_{10n+7}$ and $D_{10n+11}$ are stable derivations.
By Theorem~\ref{thm-Glanois} we know $\zeta^\fm(\{3,\bar2\}_k)$ and $\zeta^\fm(\{3,\bar2\}_k,3)$ are both unramified.
We leave the details to the interested reader.

\subsection{The inductive step for $\zeta_1^\fm(\{3,\bar2\}_k)$ and $\zeta_1^\fm(\{3,\bar2\}_k,3)$}
We first consider the case $r=10n+9<5k$. There is only one cut in $D_r$ that is nonzero.

\begin{center}
\begin{tikzpicture}[scale=0.9]
\node (A1) at (-1.55,0) {$0;$};
\node (A2) at (-1.15,0) {$0,$};
\node (A3) at (-0.75,0) {$\bar1,$};
\node (A4) at (-0.35,0) {$0,$};
\node (A5) at (0.05,0) {$0,$};
\node (A6) at (0.45,0) {$\bar1,$};
\node (A7) at (0.85,0) {$0,$};
\node (A8) at (1.25,0) {$1,$};
\node (A9) at (1.65,0) {$0,$};
\node (A9) at (2.8,0.04) {$\{S_1,S_{\bar1}\}_n,$};
\node (A10) at (4.05,0) {$0,$};
\node (A11) at (4.45,0) {$1,$};
\node (A11) at (4.85,0) {$0,$};
\node (A12) at (5.25,0) {$\bar1,$};
\node (A13) at (6.05,0) {$0,\ldots,$};
\node (A14) at (6.95,-0.5) {$0,$};
\node (A14) at (7.35,-0.5) {$\bar1,$};
\node (A14) at (7.75,-0.5) {$0,$};
\node (A15) at (8.15,-0.5) {$1,$};
\node (A16) at (8.55,-0.5) {$0,$};
\node (A17) at (8.95,-0.5) {$0;$};
\node (A17) at (11.5,-0.46) {$1\phantom{,}$ \qquad {for $\zeta_1^\fm(\{3,\bar2\}_k,3)$}};
\node (A14) at (6.95,0.5) {$0,$};
\node (A14) at (7.35,0.5) {$\bar1,$};
\node (A14) at (7.75,0.5) {$0;$};
\node (A17) at (10.7,0.52) {$1\phantom{,}$  \qquad\qquad\ \ {for $\zeta_1^\fm(\{3,\bar2\}_k)$}};

\node (C5) at (1.6,-0.4) {${}$};
\node (D4) at (1.4,0.4) {${}$};
\draw (-1.6,0.25) to (-1.6,0.4) to (D4) node {$\cic{9}$} to (4.4,0.4)  to (4.4,0.25);
\draw (-1.6,-0.25) to (-1.6,-0.4) to (C5) node {$\ccic{1\!1}$} to (5.2,-0.4)  to (5.2,-0.325);
\draw[dashed] (-0.2,0.7) to (-0.2,-0.7);
\draw[dashed] (1.8,0.7) to (1.8,-0.7);
\draw[dashed] (3.8,0.7) to (3.8,-0.7);
\draw[dashed] (5.8,0.7) to (5.8,-0.7);
\draw[dashed] (6.7,0.7) to (6.7,-0.7);
\draw[dashed] (8.7,0.0) to (8.7,-0.7);
\end{tikzpicture}
\end{center}
This occurs when the subsequence starts at the very beginning which produces the terms
\begin{align*}
D_r\zeta_1^\fm(\{3,\bar2\}_k)=\ncic{9}=&\, \zeta_1(\{3,\bar2\}_{2n+1},3)\ot \zeta^\fm(\{\bar2,3\}_{k-2-2n},\bar2)\in \calL_{10n+9}\ot \calH_{5k-10n-8}, \\
D_r\zeta_1^\fm(\{3,\bar2\}_k,3)=\ncic{9}=&\, \zeta_1(\{3,\bar2\}_{2n+1},3)\ot \zeta^\fm(\{\bar2,3\}_{k-1-2n})\in \calL_{10n+9}\ot \calH_{5k-10n-5}
\end{align*}
by induction and the proceeding section. Similarly, if $r=10n+11<5k$ then the only nontrivial cut in $D_r$
occurs when the subsequence starts at the very beginning, which produces the terms
\begin{align*}
D_r\zeta_1^\fm(\{3,\bar2\}_k)=\nncic{11}=&\, \zeta_1(\{3,\bar2\}_{2n+2})\ot \zeta^\fm(\{3,\bar2\}_{k-2-2n})\in \calL_{10n+11}\ot \calH_{5k-10n-10},\\
D_r\zeta_1^\fm(\{3,\bar2\}_k,3)=\nncic{11}=&\, \zeta_1(\{3,\bar2\}_{2n+2})\ot \zeta^\fm(\{3,\bar2\}_{k-2-2n},3)\in \calL_{10n+11}\ot \calH_{5k-10n-7}
\end{align*}
by induction and the proceeding section.
For all other odd $r$'s the proof is essentially the same as the starting $0$ does not contribute to $D_r$.
By Theorem~\ref{thm-Glanois} we see that $\zeta_1^\fm(\{3,\bar2\}_k)$ is unramified.

\subsection{The inductive step for $\zeta_1^\fm(\{\bar2,3\}_k)$ and $\zeta_1^\fm(\{\bar2,3\}_k,\bar2)$}
We now consider the pair of families of motivic Euler sums in \eqref{equ:1-2bar3}.
We first consider the case $r=10n+3$. Adding one 0 to the left of the picture in \S\ref{subsec:2bar3caser=3}
we see that there is only one extra cut in $D_r$ compared to the \eqref{equ:3bar2} case
when the subsequence starts at the very beginning. This yields the terms
\begin{equation*}
\ncic{\bar4}=\zeta_1^\fl(\{\bar2,3\}_{2n},\bar2)\ot \zeta^\fm(\{3,\bar2\}_{k-1-2n},3)
\quad\text{or} \quad
\zeta_1^\fl(\{\bar2,3\}_{2n},\bar2)\ot \zeta^\fm(\{3,\bar2\}_{k-2n})
\end{equation*}
By induction and the proceeding section (if $n=0$) we see that all the factors are unramified.
Similarly, if $r=10n+11$ then the only nontrivial cut in $D_r$
occurs when the subsequence starts at the very beginning which produces the terms
\begin{equation*}
\zeta_1^\fl(\{\bar2,3\}_{2n+2})\ot \zeta^\fm(\{\bar2,3\}_{k-2-2n})
 \quad\text{or} \quad
\zeta_1^\fl(\{\bar2,3\}_{2n+2})\ot \zeta^\fm(\{\bar2,3\}_{k-2-2n},\bar2).
\end{equation*}
Again, all the factors are unramified by induction.
For other cuts when $r=10n+3$ as well as for other odd $r$'s the proof is essentially
the same as that for $\zeta^\fm(\{\bar2,3\}_k)$ and $\zeta^\fm(\{\bar2,3\}_k,\bar2)$
since the starting $0$ does not contribute to $D_r$.
By Theorem~\ref{thm-Glanois} we see that both $\zeta_1^\fm(\{\bar2,3\}_k)$ and $\zeta_1^\fm(\{\bar2,3\}_k,\bar2)$ are unramified.

We have now completed the proof of Lemma~\ref{lem:2bar3Motivic}.

\section{Proof of Theorem~\ref{thm:2bar3Motivic}} \label{sec:proofOfThm:2bar3Motivic}
We first handle another set of base cases for arbitrary $a$.
\begin{lem}
For any integer $a\ge 0$ the MES
\begin{equation*}
\zeta_a^\fm(\bar2), \quad \zeta_a^\fm(\bar2,3), \quad \zeta_a^\fm(\bar2,3,\bar2), \quad \zeta_a^\fm(3,\bar2), \quad \zeta_a^\fm(3,\bar2,3)
\end{equation*}
are all unramified.
\end{lem}
\begin{proof}
The proof is again a standard application of Theorem~\ref{thm-Glanois} using the stable derivations $D_r$.
We leave the details to the interested reader.
\end{proof}

We now use induction to prove Theorem~\ref{thm:2bar3Motivic}. Assume $a\ge 2$, $\ell\ge 2$ and Theorem~\ref{thm:2bar3Motivic} holds if $a$ or $k$ or both are replaced by any smaller nonnegative integers.

For any odd integer $r\ge 3$ if a cut of $D_r$ starting after the first nonzero $\pm 1$ is not canceled by another cut nor does it vanish by itself then by the proof of Lemma~\ref{lem:2bar3Motivic} the left factor must have the form treated in Lemma~\ref{lem:2bar3Motivic} and the right factor one of the forms in Theorem~\ref{thm:2bar3Motivic} but with smaller $\ell$. By induction assumption $D_r$ must be stable.

If a cut starts somewhere within the first block of $a+1$ 0's then its left factor must have one of the forms in Theorem~\ref{thm:2bar3Motivic} but with $\ell<k$ and its right factor must have the form treated in Lemma~\ref{lem:2bar3Motivic}. Thus $D_r$ must be stable by induction assumption again. This completes the proof of the theorem.

\section{Proof of \eqref{equ:bar21bar2Motivic} and \eqref{equ:bar21Motivic} in Theorem \ref{thm:bar21bar2Motivic}}
Assuming Conjecture~\ref{conj:bar21bar2} holds we now show that for all $\ell\in\N$,  
\begin{align}
2^{3\ell+1}\zeta^\fm(\{\bar2,1\}_\ell,\bar2)=&\, 
        -\sum_{\ga+\gb=\ell} (-1)^\ga \zeta^\fm(3_{\ga},2,3_{\gb}),\label{equ:bar21bar2MotivicN}\\
2^{3\ell}\zeta^\fm(\{\bar2,1\}_\ell)=&\,
        \zeta^\fm(3_\ell)-2\sum_{\ga+\gb=\ell-1} (-1)^{\ga} \zeta_1^\fm(3_{\ga},2,3_{\gb}), \label{equ:bar21MotivicN}
\end{align}
by induction on $\ell$.

\subsection{Proof of base cases of \eqref{equ:bar21bar2MotivicN} and \eqref{equ:bar21MotivicN}}
If $\ell=1$ then we see easily that $D_1=0$ for all motivic Euler sums involved by Lemma~\ref{lem:D1}.
Then by the regularized stuffle relation of Euler sums
\begin{equation*}
\zeta_*(\bar2,1)=-\zeta(1,\bar2)-\zeta(\bar3)=-\frac18\zeta(3)+\frac34\zeta(3)=\frac58\zeta(3)
\end{equation*}
using Prop.~\ref{prop-singleES} and \eqref{equ:1bar2Motivic}. Thus \eqref{equ:bar21MotivicN}
follows from Theorem~\ref{thm-Glanois} when $\ell=1$:
\begin{equation}\label{equ:ell=1bar21MotivicN}
8\zeta^\fm(\bar2,1)=\zeta^\fm(3)-2\zeta_1^\fm(2)=5\zeta^\fm(3).
\end{equation}
Using the pictures \eqref{fig:bar21bar2Base}
\begin{equation}\label{fig:bar21bar2Base}
\text{
\begin{tikzpicture}[scale=0.9]
\node (RA0) at (5.05,0) {$0;$};
\node (RA1) at (5.45,0) {$1,$};
\node (RA3) at (5.85,0) {$0,$};
\node (RA3) at (6.25,0) {$\bar1,$};
\node (RB3) at (6.2,0.4) {${}$};
\node (RA4) at (6.65,0) {$\bar1,$};
\node (RD4) at (6.6,0.8) {${}$};
\node (RA5) at (7.05,0) {$0;$};
\node (RA6) at (7.45,0) {$1\phantom{,}$};
\node (RC5) at (5.8,-0.5) {${}$};
\draw (5.4,0.25) to (5.4,0.4) to (RB3) node {$\cic{2}$} to (7.0,0.4)  to (7.0,0.25);
\draw (5.8,0.25) to (5.8,0.8) to (RD4) node {$\cic{3}$} to (7.4,0.8)  to (7.4,0.25);
\draw (5,-0.25) to (5,-0.5) to (RC5) node {$\cic{1}$} to (6.6,-0.5)  to (6.6,-0.325);
\node (label) at (6.25,-1.2) {$D_3 \zeta^\fm(\bar2,1,\bar2)$};
\end{tikzpicture}
\
\begin{tikzpicture}[scale=0.9]
\node (RA0) at (5.05,0) {$0;$};
\node (RA1) at (5.45,0) {$1,$};
\node (RA3) at (5.85,0) {$0,$};
\node (RA3) at (6.25,0) {$0,$};
\node (RB3) at (6.2,0.4) {${}$};
\node (RA4) at (6.65,0) {$1,$};
\node (RD4) at (6.6,0.8) {${}$};
\node (RA5) at (7.05,0) {$0;$};
\node (RA6) at (7.45,0) {$1\phantom{,}$};
\node (RC5) at (5.8,-0.5) {${}$};
\draw (5.4,0.25) to (5.4,0.4) to (RB3) node {$\cic{2}$} to (7.0,0.4)  to (7.0,0.25);
\draw (5.8,0.25) to (5.8,0.8) to (RD4) node {$\cic{3}$} to (7.4,0.8)  to (7.4,0.25);
\draw (5,-0.25) to (5,-0.5) to (RC5) node {$\cic{1}$} to (6.6,-0.5)  to (6.6,-0.325);
\node (label) at (6.25,-1.2) {$D_3 \zeta^\fm(3,2)$};
\end{tikzpicture}
\
\begin{tikzpicture}[scale=0.9]
\node (RA0) at (5.05,0) {$0;$};
\node (RA1) at (5.45,0) {$1,$};
\node (RA3) at (5.85,0) {$0,$};
\node (RA3) at (6.25,0) {$1,$};
\node (RB3) at (6.2,0.4) {${}$};
\node (RA4) at (6.65,0) {$0,$};
\node (RD4) at (6.6,0.8) {${}$};
\node (RA5) at (7.05,0) {$0;$};
\node (RA6) at (7.45,0) {$1\phantom{,}$};
\node (RC5) at (5.8,-0.5) {${}$};
\draw (5.4,0.25) to (5.4,0.4) to (RB3) node {$\cic{2}$} to (7.0,0.4)  to (7.0,0.25);
\draw (5.8,0.25) to (5.8,0.8) to (RD4) node {$\cic{3}$} to (7.4,0.8)  to (7.4,0.25);
\draw (5,-0.25) to (5,-0.5) to (RC5) node {$\cic{1}$} to (6.6,-0.5)  to (6.6,-0.325);
\node (label) at (6.25,-1.2) {$D_3 \zeta^\fm(2,3)$};
\end{tikzpicture}
\
\begin{tikzpicture}[scale=0.9]
\node (RA1) at (5.45,0) {$0,$};
\node (RA3) at (5.85,0) {$\bar1,$};
\node (RA3) at (6.25,0) {$\bar1,$};
\node (RB3) at (6.2,-0.4) {${}$};
\node (RA4) at (6.65,0) {$0,$};
\node (RD4) at (6.6,0.4) {${}$};
\node (RA5) at (7.05,0) {$1;$};
\node (RA6) at (7.45,0) {$1\phantom{,}$};
\draw (5.8,0.25) to (5.8,0.4) to (RD4) node {$\cic{2}$} to (7.4,0.4)  to (7.4,0.25);
\draw (5.4,-0.25) to (5.4,-0.4) to (RB3) node {$\cic{1}$} to (7.0,-0.4)  to (7.0,-0.25);
\node (label) at (6.45,-1.2) {$D_3 \zeta^\fm(1,\bar2,1)$};
\end{tikzpicture}
\
\begin{tikzpicture}[scale=0.9]
\node (RA1) at (5.45,0) {$0,$};
\node (RA3) at (5.85,0) {$1,$};
\node (RA3) at (6.25,0) {$1,$};
\node (RB3) at (6.2,-0.4) {${}$};
\node (RA4) at (6.65,0) {$0,$};
\node (RD4) at (6.6,0.4) {${}$};
\node (RA5) at (7.05,0) {$0;$};
\node (RA6) at (7.45,0) {$1\phantom{,}$};
\draw (5.8,0.25) to (5.8,0.4) to (RD4) node {$\cic{2}$} to (7.4,0.4)  to (7.4,0.25);
\draw (5.4,-0.25) to (5.4,-0.4) to (RB3) node {$\cic{1}$} to (7.0,-0.4)  to (7.0,-0.25);
\node (label) at (6.45,-1.2) {$D_3 \zeta^\fm(1,3)$};
\end{tikzpicture} }
\end{equation}
together with the fact that $\zeta^\fm(\bar2)=-\zeta^\fm(2)/2$ by Prop. \ref{prop-singleES}
and \eqref{equ:ell=1bar21MotivicN}, we have
\begin{align*}
D_3 \zeta^\fm(\bar2,1,\bar2)=&\, I^\fl(0;10\bar1;\bar1)\ot I^\fm(0;\bar10;1)
    =-\frac12\zeta^\fl(\bar2,1)\ot\zeta^\fm(2)=-\frac5{16} \zeta^\fl(3)\ot\zeta^\fm(2),\\
D_3 \zeta^\fm(3,2)=&\, I^\fl(0;010;1)\ot I^\fm(0;10;1)=\zeta_1^\fl(2)\ot\zeta^\fm(2)=-2\zeta^\fl(3)\ot\zeta^\fm(2),\\
D_3 \zeta^\fm(2,3)=&\, \big(I^\fl(1;010;0)+I^\fl(0;100;1)\big)\ot I^\fm(0;10;1)=3\zeta^\fl(3)\ot\zeta^\fm(2).
\end{align*}
Hence
\begin{equation*}
   16 D_3 \zeta^\fm(\bar2,1,\bar2)=D_3\big( \zeta^\fm(3,2)- \zeta^\fm(2,3) \big).
\end{equation*}
For \eqref{equ:12bar1MotivicN}  it is easy to see from the last two pictures in \eqref{fig:bar21bar2Base} that
$$
D_r \zeta^\fm(1,\bar2,1)=D_r \zeta^\fm(1,3)=0 \quad \text{for } r=1,3.
$$
Hence, the $\ell=1$ case of  \eqref{equ:bar21bar2MotivicN} and \eqref{equ:bar21MotivicN} follows from Theorem~\ref{thm-Glanois}.

We now assume \eqref{equ:bar21bar2MotivicN} and \eqref{equ:bar21MotivicN} holds for all positive integers $\ell<k$.

\subsection{Inductive proof of \eqref{equ:bar21bar2MotivicN}}
We first show that
\begin{equation}\label{equ:D_rbar21bar2}
2^{3\ell+1}D_r \zeta^\fm(\{\bar2,1\}_k,\bar2)= -\sum_{\ga+\gb=k} (-1)^\ga D_r \zeta^\fm(3_{\ga},2,3_{\gb})
\end{equation}
for all positive integers $r=6n+3, 6n+5, 6n+7<3k+2$.

\subsubsection{$r=6n+3$} \label{sec:D_6n+3-3a23b}
For $r=6n+3$ we have the following picture when $k$ is even (the case of odd $k$ can be dealt with similarly).
\begin{equation}\label{fig:D_6n+3-bar21bar2}
\text{
\begin{tikzpicture}[scale=0.9]
\node (A1) at (-3.8,0) {$0;$};
\node (A12) at (-3.4,0) {$\bar1,$};
\node (A13) at (-3,0) {$0,$};
\node (A4) at (-2.6,0) {$1,$};
\node (A2) at (-2.2,0) {$1,$};
\node (A3) at (-1.8,0) {$0,$};
\draw[dashed] (-1.6,0.7) to (-1.6,-0.7);
\node (A4) at (-1.4,0) {$\bar1,$};
\node (A5) at (-1,0) {$\bar1,$};
\node (A6) at (-0.6,0) {$0,$};
\draw[dashed] (-0.2,0.7) to (-0.2,-0.7);
\node (A7) at (0.2,0) {$\cdots\!,$};
\draw[dashed] (0.6,0.7) to (0.6,-0.7);
\node (A8) at (0.85,0) {$1,$};
\node (A8) at (1.25,0) {$1,$};
\node (A9) at (1.65,0) {$0,$};
\node (A9) at (2.8,0.04) {$\{\bar1\bar10110\}_n$};
\node (A10) at (4.05,0) {$\bar1,$};
\node (A11) at (4.45,0) {$\bar1,$};
\node (A11) at (4.85,0) {$0,$};
\node (A12) at (5.25,0) {$1,$};
\node (A13) at (5.65,0) {$1,$};
\node (A13) at (6.45,0) {$0,\cdots\!,$};
\node (A14) at (7.35,0) {$\bar1,$};
\node (A14) at (7.75,0) {$\bar1,$};
\node (A14) at (8.15,0) {$0;$};
\node (A15) at (8.55,0) {$1\phantom{,}$};
\node (C1) at (2.6,-0.4) {${}$};
\node (D4) at (3,0.6) {${}$};
\node (D3) at (3.4,-0.6) {${}$};
\draw (0.8,-0.25) to (0.8,-0.4) to (C1) node {$\cic{1}$} to (4.4,-0.4)  to (4.4,-0.325);
\draw (1.2,0.25) to (1.2,0.6) to (D4) node {$\cic{2}$} to (4.8,0.6)  to (4.8,0.25);
\draw (1.6,-0.25) to (1.6,-0.6) to (D3) node {$\cic{3}$} to (5.2,-0.6)  to (5.2,-0.25);
\draw[dashed] (1.8,0.7) to (1.8,-0.7);
\draw[dashed] (3.8,0.7) to (3.8,-0.7);
\draw[dashed] (6.2,0.7) to (6.2,-0.7);
\draw[dashed] (5,0.7) to (5,-0.7);
\draw[dashed] (7.1,0.7) to (7.1,-0.7);
\draw[dashed] (8.3,0.7) to (8.3,-0.7);
\node (C1) at (2.5,-1.2) {Possible cuts of $D_{6n+3}\zeta^\fm(\{\bar2,1\}_k,\bar2)$};
\end{tikzpicture}}
\end{equation}
Note that in general $\ncic{2}+\ncic{3}=0$ and $\ncic{1}\,${} is antisymmetric
so that they all disappear, except for the following special case that cannot be canceled:
a type  $\ncic{1}\,${} cut starts at the initial $0$. Thus
\begin{align}
D_{6n+3}\zeta^\fm(\{\bar2,1\}_k,\bar2)
=&\,  I^\fl(0;10\{\bar1\bar10110\}_n\bar1;\bar1)\ot I^\fm(0;10\bar1\bar10\{110\bar1\bar10\}_{k/2-n-1};1)  \notag\\
=&\, \zeta^\fl(\{\bar2,1\}_{2n+1}) \ot \zeta^\fm(\{\bar2,1\}_{k-2n-1},\bar2). \label{equ:inductUsebar21bar2}
\end{align}

We note that there are two kinds of cuts for $D_r \zeta^\fm(3_{\ga},2,3_{\gb})$:
(i) those that don't involve the component $\rho(2)$ and (ii) those that do.

\begin{equation}\label{fig:D_6n+3-3a23bType(i)}
\text{
\begin{tikzpicture}[scale=0.9]
\node (A13) at (-3,0) {$0;$};
\node (A4) at (-2.6,0) {$1,$};
\node (A2) at (-2.2,0) {$0,$};
\node (A3) at (-1.8,0) {$0,$};
\draw[dashed] (-1.6,0.7) to (-1.6,-0.7);
\node (A4) at (-1.4,0) {$1,$};
\node (A5) at (-1,0) {$0,$};
\node (A6) at (-0.6,0) {$0,$};
\draw[dashed] (-0.3,0.7) to (-0.3,-0.7);
\node (A7) at (0.1,0) {$\cdots\!,$};
\draw[dashed] (0.5,0.7) to (0.5,-0.7);
\node (A8) at (0.7,0) {$1,$};
\node (A8) at (1.1,0) {$0,$};
\node (A9) at (1.5,0) {$0,$};
\node (A9) at (2.7,0.04) {$\{100100\}_n$};
\node (A10) at (4.05,0) {$1,$};
\node (A11) at (4.45,0) {$0,$};
\node (A11) at (4.85,0) {$0,$};
\node (A12) at (5.25,0) {$1,$};
\node (A13) at (5.65,0) {$0,$};
\node (A13) at (6.45,0) {$0,\cdots\!,$};
\node (A14) at (7.35,0) {$1,$};
\node (A14) at (7.75,0) {$0,$};
\node (A14) at (8.15,0) {$0;$};
\node (A15) at (8.55,0) {$1\phantom{,}$};
\node (C1) at (3.6,0.4) {${}$};
\node (D2) at (3.2,-0.4) {${}$};
\draw (2,0.25) to (2,0.4) to (C1) node {$\cic{2}$} to (5.6,0.4)  to (5.6,0.25);
\draw (1.5,-0.25) to (1.5,-0.4) to (D2) node {$\cic{1}$} to (5.2,-0.4)  to (5.2,-0.25);
\draw[dashed] (1.7,0.7) to (1.7,-0.7);
\draw[dashed] (3.8,0.7) to (3.8,-0.7);
\draw[dashed] (6.25,0.7) to (6.25,-0.7);
\draw[dashed] (5.05,0.7) to (5.05,-0.7);
\draw[dashed] (7.1,0.7) to (7.1,-0.7);
\draw[dashed] (8.35,0.7) to (8.35,-0.7);
\node (C1) at (2.5,-1.2) {Type (i) cuts of $D_{6n+3}\zeta^\fm(3_{\ga},2,3_{\gb})$};
\end{tikzpicture}}
\end{equation}
Consider the picture above for cuts of type (i). Then $\ncic{1}\,${} is always canceled by $\ncic{2}\,${}
except when $\ncic{1}\,${} ends at the last 1. Type (i) has only one nontrivial contribution
\begin{equation}\label{equ:type(i)Contr}
\zeta^\fl(3_{2n+1})\ot \zeta^\fm(3_{\ga-2n-1},2,3_{\gb}).
\end{equation}

For type (ii) cuts shown below

\begin{equation}\label{fig:D_6n+3-3a23b}
\text{
\begin{tikzpicture}[scale=0.9]
\node (A13) at (-3,0) {$0;$};
\node (A4) at (-2.6,0) {$1,$};
\node (A2) at (-2.2,0) {$0,$};
\node (A3) at (-1.8,0) {$0,$};
\draw[dashed] (-1.6,0.7) to (-1.6,-0.7);
\node (A4) at (-1.4,0) {$1,$};
\node (A5) at (-1,0) {$0,$};
\node (A6) at (-0.6,0) {$0,$};
\draw[dashed] (-0.2,0.7) to (-0.2,-0.7);
\node (A7) at (0.2,0) {$\cdots\!,$};
\draw[dashed] (0.6,0.7) to (0.6,-0.7);
\node (A8) at (0.85,0) {$1,$};
\node (A8) at (1.25,0) {$0,$};
\node (A9) at (1.65,0) {$0,$};
\node (A9) at (2.8,0.04) {$\{1,0,0\}_i,$};
\node (A10) at (4.05,0) {$1,$};
\node (A11) at (4.45,0) {$0,$};
\node (A9) at (5.65,0.04) {$\{1,0,0\}_j,$};
\node (A12) at (6.85,0) {$1,$};
\node (A13) at (7.25,0) {$0,$};
\node (A13) at (8.05,0) {$0,\cdots\!,$};
\node (A14) at (8.85,0) {$1,$};
\node (A14) at (9.25,0) {$0,$};
\node (A14) at (9.65,0) {$0;$};
\node (A15) at (10.05,0) {$1\phantom{,}$};
\node (C1) at (2.9,-0.4) {${}$};
\node (D2) at (3.5,0.4) {${}$};
\node (D3) at (2.8,0.6) {${}$};
\node (D4) at (5.0,-0.6) {${}$};
\node (D5) at (5.6,0.6) {${}$};
\draw (0.8,-0.25) to (0.8,-0.4) to (C1) node {$\cic{1}$} to (6.0,-0.4)  to (6.0,-0.325);
\draw (1.2,0.25) to (1.2,0.4) to (D2) node {$\cic{2}$} to (6.8,0.4)  to (6.8,0.25);
\draw (4,-0.25) to (4,-0.6) to (D4) node {$\cic{3}$} to (6.0,-0.6)  to (6.0,-0.25);
\draw (4.4,0.25) to (4.4,0.6) to (D5) node {$\cic{4}$} to (6.8,0.6)  to (6.8,0.25);
\draw[dashed] (1.8,0.7) to (1.8,-0.7);
\draw[dashed] (3.8,0.7) to (3.8,-0.7);
\draw[dashed] (4.6,0.7) to (4.6,-0.7);
\draw[dashed] (6.6,0.7) to (6.6,-0.7);
\draw[dashed] (7.9,0.7) to (7.9,-0.7);
\draw[dashed] (8.7,0.7) to (8.7,-0.7);
\draw[dashed] (9.84,0.7) to (9.84,-0.7);
\node (C1) at (3.5,-1.2) {Type (ii) cuts $D_{6n+3}\zeta^\fm(3_{\ga},2,3_{\gb}), i+j=2n$, starting before $\rho(2)$};
\end{tikzpicture}}
\end{equation}
we see that $\ncic{1}\,${} and $\ncic{2}\,${} contribute to
\begin{equation*}
    \sum_{i+j=2n} \big(\zeta_1^\fl(3_i,2,3_j)-\zeta_1^\fl(3_{j-1},2,3_{i+1}) \big)
 \ot \zeta^\fm(3_{\ga-i-1},2,3_{\gb-j}).
\end{equation*}
When $j=0$ the second term inside the left factor means
$$
-\zeta_1^\fl(3_{-1},2,3_{2n+1}) \ot \zeta^\fm(3_{\ga-2n-1},2,3_{\gb}):=-\zeta^\fl(3_{2n+1})\ot \zeta^\fm(3_{\ga-2n-1},2,3_{\gb}),
$$
which is canceled by \eqref{equ:type(i)Contr}.
There are two more special cases when a type (ii) cut starts in $\rho(2)$ (consider
$j=2n+1$ in \eqref{fig:D_6n+3-3a23b})
\begin{equation*}
    \ncic{3}+ \ncic{4}= \big(\zeta^\fl(3_{2n+1})-\zeta_1^\fl(3_{2n},2) \big)
 \ot \zeta^\fm(3_{\ga},2,3_{\gb-2n-1}).
\end{equation*}
Putting all the above together we obtain
\begin{align*}
D_{6n+3}\zeta^\fm(3_{\ga},2,3_{\gb})=&\, \sum_{\substack{i+j=2n\\ a-1\ge i\ge 0, b\ge j\ge 1}} \big(\zeta_1^\fl(3_i,2,3_j)-\zeta_1^\fl(3_{j-1},2,3_{i+1}) \big)
 \ot \zeta^\fm(3_{\ga-i-1},2,3_{\gb-j}) \\
&\, + \gd_{\ga\ge 2n+1} \zeta_1^\fl(3_{2n},2) \ot \zeta^\fm(3_{\ga-2n-1},2,3_{\gb}) \\
&\, +\gd_{\gb\ge 2n+1} \big(\zeta^\fl(3_{2n+1})-\zeta_1^\fl(3_{2n},2) \big)
 \ot \zeta^\fm(3_{\ga},2,3_{\gb-2n-1}).
\end{align*}
Therefore
\begin{align*}
&\, \sum_{\ga+\gb=k} (-1)^\ga D_{6n+3}\zeta^\fm(3_{\ga},2,3_{\gb}) \\
=&\, \sum_{\substack{\ga+\gb=k,i+j=2n\\ \ga-1\ge i\ge 0, \gb\ge j\ge 1}} (-1)^\ga \big(\zeta_1^\fl(3_i,2,3_j)-\zeta_1^\fl(3_{j-1},2,3_{i+1}) \big)
 \ot \zeta^\fm(3_{\ga-i-1},2,3_{\gb-j}) \\
&\,  +\sum_{\ga+\gb=k}  (-1)^\ga \gd_{\ga\ge 2n+1} \zeta_1^\fl(3_{2n},2) \ot \zeta^\fm(3_{\ga-2n-1},2,3_{\gb}) \\
&\,  +\sum_{\ga+\gb=k} (-1)^\ga \gd_{\gb\ge 2n+1} \big(\zeta^\fl(3_{2n+1})-\zeta_1^\fl(3_{2n},2) \big)
 \ot \zeta^\fm(3_{\ga},2,3_{\gb-2n-1})\\
=&\, \sum_{\substack{i+j=2n,i\ge0,j\ge1\\ \ga+\gb=k-2n-1}} (-1)^{a+i+1} \big(\zeta_1^\fl(3_i,2,3_j)-\zeta_1^\fl(3_{j-1},2,3_{i}) \big)
 \ot \zeta^\fm(3_{\ga},2,3_{\gb}) \\
&\, +\sum_{\ga+\gb=k-2n-1}(-1)^\ga  \big(\zeta^\fl(3_{2n+1})-2\zeta_1^\fl(3_{2n},2) \big)
 \ot \zeta^\fm(3_{\ga},2,3_{\gb})\\
=&\, -2\sum_{\substack{i+j=2n,i\ge0,j\ge1\\ \ga+\gb=k-2n-1}} (-1)^{a+i} \zeta_1^\fl(3_i,2,3_j)
 \ot \zeta^\fm(3_{\ga},2,3_{\gb}) \\
&\, +\sum_{\ga+\gb=k-2n-1}(-1)^\ga   \big(\zeta^\fl(3_{2n+1})-2\zeta_1^\fl(3_{2n},2) \big)
 \ot \zeta^\fm(3_{\ga},2,3_{\gb}) \\
=&\,\left(\zeta^\fl(3_{2n+1})-2\sum_{i+j=2n} (-1)^{i} \zeta_1^\fl(3_i,2,3_j) \right)
 \ot  \sum_{\ga+\gb=k-2n-1} (-1)^\ga  \zeta^\fm(3_{\ga},2,3_{\gb})  \\
= &\,\zeta^\fl(\{\bar2,1\}_{2n+1}) \ot \zeta^\fm(\{\bar2,1\}_{k-2n-1},\bar2)
\end{align*}
 by inductive assumption. By comparing with  \eqref{equ:inductUsebar21bar2}
we see that \eqref{equ:D_rbar21bar2} holds for $r=6n+3$.

\subsubsection{$r=6n+5$}\label{sec:D_6n+5-bar21bar2}
It is obvious that if $r=6n+5$ then the left-hand side of \eqref{equ:D_rbar21bar2} vanishes
since $\rho(\{\bar2,1\}_k,\bar2)$
has period 6 so that every cut of $D_{6n+5}$ starts and ends with the same number.

Similar to the case $r=6n+3$, there are two kinds of cuts for $D_{6n+5}\zeta^\fm(3_{\ga},2,3_{\gb})$:
(i) those involve the component $\rho(2)$ and (ii) those don't. All cuts in (ii)
together with the cut starting at the 0 in $\rho(2)$ clearly vanish by periodicity.
For type (ii) cuts we consider the following picture

\begin{center}
\begin{tikzpicture}[scale=0.9]
\node (A13) at (-3,0) {$0;$};
\node (A4) at (-2.6,0) {$1,$};
\node (A2) at (-2.2,0) {$0,$};
\node (A3) at (-1.8,0) {$0,$};
\draw[dashed] (-1.6,0.7) to (-1.6,-0.7);
\node (A4) at (-1.4,0) {$1,$};
\node (A5) at (-1,0) {$0,$};
\node (A6) at (-0.6,0) {$0,$};
\draw[dashed] (-0.2,0.7) to (-0.2,-0.7);
\node (A7) at (0.2,0) {$\cdots\!,$};
\draw[dashed] (0.6,0.7) to (0.6,-0.7);
\node (A8) at (0.85,0) {$1,$};
\node (A8) at (1.25,0) {$0,$};
\node (A9) at (1.65,0) {$0,$};
\node (A9) at (2.8,0.04) {$\{1,0,0\}_i,$};
\node (A10) at (4.05,0) {$1,$};
\node (A11) at (4.45,0) {$0,$};
\node (A9) at (5.65,0.04) {$\{1,0,0\}_j,$};
\node (A12) at (6.85,0) {$1,$};
\node (A13) at (7.25,0) {$0,$};
\node (A13) at (8.05,0) {$0,\cdots\!,$};
\node (A14) at (8.85,0) {$1,$};
\node (A14) at (9.25,0) {$0,$};
\node (A14) at (9.65,0) {$0;$};
\node (A15) at (10.05,0) {$1\phantom{,}$};

\node (C1) at (3.7,0.4) {${}$};
\node (D2) at (4.2,-0.4) {${}$};
\draw (2.2,0.25) to (2.2,0.4) to (C1) node {$\cic{2}$} to (7.2,0.4)  to (7.2,0.25);
\draw (1.6,-0.25) to (1.6,-0.4) to (D2) node {$\cic{1}$} to (6.8,-0.4)  to (6.8,-0.25);
\draw[dashed] (1.85,0.7) to (1.85,-0.7);
\draw[dashed] (3.8,0.7) to (3.8,-0.7);
\draw[dashed] (4.65,0.7) to (4.65,-0.7);
\draw[dashed] (6.6,0.7) to (6.6,-0.7);
\draw[dashed] (7.9,0.7) to (7.9,-0.7);
\draw[dashed] (8.7,0.7) to (8.7,-0.7);
\draw[dashed] (9.84,0.7) to (9.84,-0.7);
\node (C1) at (3.25,-1.2) {Type (i) cuts of $D_{6n+5}\zeta^\fm(3_{\ga},2,3_{\gb}), i+j=2n+1$, starting before 0 in $\rho(2)=10$};
\end{tikzpicture}
\end{center}
Then we see immediately that cut  $\ncic{1}\,${} with index $(i,j)$ shown as in the picture is canceled by
the cut $\ncic{2}\,${} with index $(j,i)$ by path reversal. Note that we allow $\ncic{2}\,${} to start at the number 1 in $\rho(2)=10$
but not after.

To summarize, \eqref{equ:D_rbar21bar2} holds (actually vanishes on both sides) for $r=6n+5$.

\subsubsection{$r=6n+7$}
By periodicity, each cut of $D_{6n+7}$ straddling over $\rho(2)$ must start and end with the same number and must vanish.
For all the other cut we may use the pictures \eqref{fig:r=6n+7-bar21} and \eqref{fig:r=6n+7-333} to conclude that
all terms in \eqref{equ:bar21MotivicN} vanish under the derivation $D_{6n+7}$.

In summary, the equation \eqref{equ:D_rbar21bar2} holds for all odd $r<3\ell$. By Lemma~\ref{lem:D1} and Theorem~\ref{thm-Glanois}
we see that  there is some $c_\ell\in\Q$ such that
\begin{equation*}
2^{3\ell+1}\zeta^\fm(\{\bar2,1\}_k,\bar2)=c_k  \zeta^\fm(3k+2)-\sum_{\ga+\gb=k} (-1)^\ga \zeta^\fm(3_{\ga},2,3_{\gb})
\end{equation*}
By applying the period map \eqref{equ:periodMap} we see that $c_k=0$. This
completes our inductive proof of \eqref{equ:bar21bar2MotivicN}, namely, \eqref{equ:bar21bar2Motivic}.

\subsection{Inductive proof of \eqref{equ:bar21MotivicN}}
 We now prove that
\begin{equation}\label{equ:D_rbar21k}
 2^{3\ell}D_r \zeta^\fm(\{\bar2,1\}_k)=
D_r \zeta^\fm(3_k)-2\sum_{\ga+\gb=k-1} (-1)^{\ga} D_r \zeta_1^\fm(3_{\ga},2,3_{\gb})
\end{equation}
for all positive integers $r=6n+3, 6n+5, 6n+7<3k$.

\subsubsection{$r=6n+3$}\label{sec:D_6n+3-bar21}
For $r=6n+3$ we get the following picture when $k$ is even (the case of odd $k$ can be dealt with similarly)
by inserting another 1 to the left of the last 1
and removing the $\bar101$ to the right of the initial 0 in picture \eqref{fig:D_6n+3-bar21bar2}:
\begin{equation}\label{fig:D_6n+3-bar21}
\text{
\begin{tikzpicture}[scale=0.9]
\node (A4) at (-2.6,0) {$0;$};
\node (A2) at (-2.2,0) {$1,$};
\node (A3) at (-1.8,0) {$0,$};
\draw[dashed] (-1.6,0.7) to (-1.6,-0.7);
\node (A4) at (-1.4,0) {$\bar1,$};
\node (A5) at (-1,0) {$\bar1,$};
\node (A6) at (-0.6,0) {$0,$};
\draw[dashed] (-0.3,0.7) to (-0.3,-0.7);
\node (A7) at (0.1,0) {$\cdots\!,$};
\draw[dashed] (0.5,0.7) to (0.5,-0.7);
\node (A8) at (0.7,0) {$1,$};
\node (A8) at (1.1,0) {$1,$};
\node (A9) at (1.5,0) {$0,$};
\node (A9) at (2.75,0.04) {$\{\bar1\bar10110\}_n$};
\node (A10) at (4.05,0) {$\bar1,$};
\node (A11) at (4.45,0) {$\bar1,$};
\node (A11) at (4.85,0) {$0,$};
\node (A12) at (5.25,0) {$1,$};
\node (A13) at (5.65,0) {$1,$};
\node (A13) at (6.45,0) {$0,\cdots\!,$};
\node (A14) at (7.35,0) {$\bar1,$};
\node (A14) at (7.75,0) {$\bar1,$};
\node (A14) at (8.15,0) {$0,$};
\node (A15) at (8.55,0) {$1;$};
\node (A15) at (8.95,0) {$1\phantom{,}$};
\node (C1) at (2.6,-0.4) {${}$};
\node (D4) at (3,0.6) {${}$};
\node (D3) at (3.4,-0.6) {${}$};
\draw (0.65,-0.25) to (0.65,-0.4) to (C1) node {$\cic{1}$} to (4.4,-0.4)  to (4.4,-0.325);
\draw (1.05,0.25) to (1.05,0.6) to (D4) node {$\cic{2}$} to (4.8,0.6)  to (4.8,0.25);
\draw (1.45,-0.25) to (1.45,-0.6) to (D3) node {$\cic{3}$} to (5.2,-0.6)  to (5.2,-0.25);
\draw[dashed] (1.75,0.7) to (1.75,-0.7);
\draw[dashed] (3.8,0.7) to (3.8,-0.7);
\draw[dashed] (6.25,0.7) to (6.25,-0.7);
\draw[dashed] (5.05,0.7) to (5.05,-0.7);
\draw[dashed] (7.1,0.7) to (7.1,-0.7);
\draw[dashed] (8.35,0.7) to (8.35,-0.7);
\node (C1) at (2.25,-1.2) {Possible cuts of $D_{6n+3}\zeta^\fm(\{\bar2,1\}_k)$};
\end{tikzpicture}}
\end{equation}
Using the same argument as for \eqref{fig:D_6n+3-bar21bar2} we get
\begin{equation} \label{equ:D_6n+3-bar21k}
D_{6n+3}\zeta^\fm(\{\bar2,1\}_k)
=\zeta^\fl(\{\bar2,1\}_{2n+1}) \ot \zeta^\fm(\{\bar2,1\}_{k-2n-1}).
\end{equation}

On the other hand, by \eqref{equ:D_r1bar2}
\begin{equation*}
 D_{6n+3} \zeta^\fm(3_k)= \zeta^\fl(3_{2n+1})\ot \zeta^\fm(3_{k-2n-1}).
\end{equation*}

Further, for any fixed $\ga+\gb=k-1$, $D_{6n+3}\zeta_1^\fm(3_{\ga},2,3_{\gb})$ can be computed
similarly as in \S\ref{sec:D_6n+3-3a23b} according to whether the cuts involve the component $\rho(2)$ or not.
By adding one more 0 at the beginning of pictures \eqref{fig:D_6n+3-3a23bType(i)} and  \eqref{fig:D_6n+3-3a23b}
we find that there will be one more term of cut $\ncic{2}\,${} starting at the very first 0 giving rise to
the additional term
\begin{equation*}
 \zeta_1^\fl(3_{\ga},2,3_{2n-\ga})\ot   \zeta^\fm(3_{k-2n-1}).
\end{equation*}
Therefore,  by induction assumption,
\begin{align*}
&\, \sum_{\ga+\gb=k-1} (-1)^\ga D_{6n+3}\zeta_1^\fm(3_{\ga},2,3_{\gb}) \\
= &\,
\left(\zeta^\fl(3_{2n+1})-2\sum_{i+j=2n} (-1)^{i} \zeta_1^\fl(3_i,2,3_j) \right)
 \ot \sum_{\ga+\gb=k-2n-2} (-1)^\ga \zeta_1^\fm(3_{\ga},2,3_{\gb})  \\
&\,+\sum_{\ga+\gb=2n} (-1)^\ga  \zeta_1^\fl(3_{\ga},2,3_{\gb})\ot   \zeta^\fm(3_{k-2n-1}) \\
=&\, 8^{2n+1}\zeta^\fl(\{\bar2,1\}_{2n+1}) \ot \sum_{\ga+\gb=k-2n-2} (-1)^\ga \zeta_1^\fm(3_{\ga},2,3_{\gb})
+\sum_{\ga+\gb=2n} (-1)^\ga  \zeta_1^\fl(3_{\ga},2,3_{\gb})\ot   \zeta^\fm(3_{k-2n-1}).
\end{align*}
Thus by induction again,
\begin{align*}
&\, D_{6n+3}\zeta^\fm(3_k)-2\sum_{\ga+\gb=k-1} (-1)^\ga D_{6n+3}\zeta_1^\fm(3_{\ga},2,3_{\gb}) \\
= &\,
8^{2n+1} \zeta^\fl(\{\bar2,1\}_{2n+1}) \ot \left(-2\sum_{\ga+\gb=k-2n-2} (-1)^\ga \zeta_1^\fm(3_{\ga},2,3_{\gb})\right)
+8^{2n+1} \zeta^\fl(\{\bar2,1\}_{2n+1})   \ot   \zeta^\fm(3_{k-2n-1}) \\
=&\,8^k \zeta^\fl(\{\bar2,1\}_{2n+1}) \ot \zeta_1^\fm(\{\bar2,1\}_{k-2n-1}).
\end{align*}
Comparing to \eqref{equ:D_6n+3-bar21k} we see immediately that \eqref{equ:D_rbar21k} holds for $r=6n+3$.

\subsubsection{$r=6n+5$}
Exactly the same argument as in \S\ref{sec:D_6n+5-bar21bar2} works almost word-for-word, implying that
both sides of \eqref{equ:D_rbar21k} vanish if $r=6n+5$.

\subsubsection{$r=6n+7$}\label{sec:D_6n+7-bar21}
By modifying the pictures \eqref{fig:r=6n+7-bar21} and \eqref{fig:r=6n+7-333}
we find that \eqref{equ:D_rbar21k} holds for $r=6n+7$.

\medskip
Combining all the results in \S\ref{sec:D_6n+3-bar21}-\S\ref{sec:D_6n+7-bar21} we see that
by Lemma~\ref{lem:D1} and Theorem~\ref{thm-Glanois} there is some $d_\ell\in\Q$ such that
\begin{equation*} 
 2^{3\ell}D_r \zeta^\fm(\{\bar2,1\}_k)=d_\ell  \zeta^\fm(3k)+
D_r \zeta^\fm(3_k)-2\sum_{\ga+\gb=k-1} (-1)^{\ga} D_r \zeta_1^\fm(3_{\ga},2,3_{\gb})
\end{equation*}
By applying the period map \eqref{equ:periodMap} we see that $d_k=0$. This
completes our inductive proof of \eqref{equ:bar21MotivicN}, namely, \eqref{equ:bar21Motivic}.

\section{Proof of \eqref{equ:12bar1Motivic} in Theorem \ref{thm:bar21bar2Motivic}}
Assuming Conjecture~\ref{conj:bar21bar2} holds we finally show that for all $\ell\in\N$,  
\begin{equation} \label{equ:12bar1MotivicN}
 2^{3\ell-1}\zeta^\fm(\{1,\bar2\}_\ell,1)= -3\sum_{\ga+\gb=\ell, 2\nmid \gb}\zeta^\fm(3_{\ga},1,3_{\gb})
\end{equation}
by using induction on $\ell$.

\subsection{The base case}
If $\ell=1$ then clearly $D_1 \zeta^\fm(1,\bar2,1)=D_1 \zeta^\fm(1,3)=0$ by Lemma~\ref{lem:D1}.
To consider the action of the derivation $D_3$ on \eqref{equ:12bar1MotivicN} we may look at the following pictures:
\begin{center}
\begin{tikzpicture}[scale=0.9]
\node (A0) at (0.05,0) {$0;$};
\node (A1) at (0.45,0) {$\bar1,$};
\node (A2) at (0.85,0) {$\bar1,$};
\node (A3) at (1.25,0) {$0,$};
\node (A4) at (1.65,0) {$1;$};
\node (A5) at (2.05,0) {$1\phantom{,}$};
\node (B3) at (1.2,0.4) {${}$};
\node (D4) at (0.8,-0.4) {${}$};
\draw (0.4,0.25) to (0.4,0.4) to (B3) node {$\cic{2}$} to (2.0,0.4)  to (2.0,0.25);
\draw (0,-0.25) to (0,-0.4) to (D4) node {$\cic{1}$} to (1.6,-0.4)  to (1.6,-0.25);
\node (lab) at (.7,-1) {$D_3\zeta^\fm(1,\bar2,1)=0$};
\end{tikzpicture}
\qquad
\begin{tikzpicture}[scale=0.9]
\node (A0) at (0.05,0) {$0;$};
\node (A1) at (0.45,0) {$1,$};
\node (A2) at (0.85,0) {$1,$};
\node (A3) at (1.25,0) {$0,$};
\node (A4) at (1.65,0) {$0;$};
\node (A5) at (2.05,0) {$1\phantom{,}$};
\node (lab) at (1,-1) {$D_3\zeta^\fm(1,3)=0$};
\end{tikzpicture}
\end{center}
We find that $D_3\zeta^\fm(1,3)=0$ since there is no nonzero cut, $\ncic{2}=0$ by anti-symmetry and
$$D_3\zeta^\fm(1,\bar2,1)=\ncic{1}=I^\fl(0;\bar1\bar10;1)\ot I^\fm(0;1;1)=0.
$$
Hence by Lemma~\ref{lem:D1} and Theorem~\ref{thm-Glanois} there is some $e_1\in\Q$ such that $4\zeta^\fm(1,\bar2,1)=e_1\zeta^\fm(4)-3\zeta^\fm(1,3)$
and thus \eqref{equ:12bar1MotivicN} holds for $\ell=1$.

In fact, by the table of values contained in \cite[Appendix E.1]{ZhaoBook},
$\zeta(1,3)=-16A-4B$ where $A=\zeta(1,1,\bar2)$ and $B=\zeta(2,\bar2)$. Further, the stuffle regularized value
\begin{equation*}
\zeta_*(1,\bar2,1)=-2\zeta(1,1,\bar2)-\zeta(2,\bar2)-\zeta(1,\bar3)=-12A-3B.
\end{equation*}
Hence
\begin{equation}\label{equ:1bar21}
4\zeta_*(1,\bar2,1)=-3\zeta(1,3).
\end{equation}
By applying the period and using \eqref{equ:1bar21} we get $e_1=0$ .

\subsection{The inductive step}
We now assume that \eqref{equ:12bar1MotivicN} holds for $\ell<k$ for some $k\ge 2$. We will now prove
that
\begin{equation}\label{equ:D_r1bar2k1}
2^{3\ell-1}D_r\zeta^\fm(\{1,\bar2\}_k,1)= -3\sum_{\ga+\gb=k, 2\nmid \gb}D_r\zeta^\fm(3_{\ga},1,3_{\gb})
\end{equation}
for all positive integers $r=6n+3, 6n+5, 6n+7<3k$.

\subsubsection{$r=6n+3$}\label{sec:D_6n+3-1bar2k1}
For $r=6n+3$ we get the following picture when $k$ is even (the case of odd $k$ can be dealt with similarly)
by inserting another 1 to the right of the initial 0 in picture \eqref{fig:D_6n+3-bar21}:
\begin{equation}\label{fig:D_6n+3bar21}
\text{
\begin{tikzpicture}[scale=0.9]
\node (A4) at (-3,0) {$0;$};
\node (A4) at (-2.6,0) {$1,$};
\node (A2) at (-2.2,0) {$1,$};
\node (A3) at (-1.8,0) {$0,$};
\draw[dashed] (-1.6,0.7) to (-1.6,-0.7);
\node (A4) at (-1.4,0) {$\bar1,$};
\node (A5) at (-1,0) {$\bar1,$};
\node (A6) at (-0.6,0) {$0,$};
\draw[dashed] (-0.3,0.7) to (-0.3,-0.7);
\node (A7) at (0.1,0) {$\cdots\!,$};
\draw[dashed] (0.45,0.7) to (0.45,-0.7);
\node (A8) at (0.7,0) {$1,$};
\node (A8) at (1.1,0) {$1,$};
\node (A9) at (1.5,0) {$0,$};
\node (A9) at (2.75,0.04) {$\{\bar1\bar10110\}_n$};
\node (A10) at (4.05,0) {$\bar1,$};
\node (A11) at (4.45,0) {$\bar1,$};
\node (A11) at (4.85,0) {$0,$};
\node (A12) at (5.25,0) {$1,$};
\node (A13) at (5.65,0) {$1,$};
\node (A13) at (6.45,0) {$0,\cdots\!,$};
\node (A14) at (7.35,0) {$\bar1,$};
\node (A14) at (7.75,0) {$\bar1,$};
\node (A14) at (8.15,0) {$0,$};
\node (A15) at (8.55,0) {$1;$};
\node (A15) at (8.95,0) {$1\phantom{,}$};
\node (C1) at (2.6,-0.4) {${}$};
\node (D4) at (2.9,0.6) {${}$};
\node (D3) at (3.3,-0.6) {${}$};
\draw (0.65,-0.25) to (0.65,-0.4) to (C1) node {$\cic{1}$} to (4.4,-0.4)  to (4.4,-0.325);
\draw (1.05,0.25) to (1.05,0.6) to (D4) node {$\cic{2}$} to (4.8,0.6)  to (4.8,0.25);
\draw (1.45,-0.25) to (1.45,-0.6) to (D3) node {$\cic{3}$} to (5.2,-0.6)  to (5.2,-0.25);
\draw[dashed] (1.7,0.7) to (1.7,-0.7);
\draw[dashed] (3.8,0.7) to (3.8,-0.7);
\draw[dashed] (6.25,0.7) to (6.25,-0.7);
\draw[dashed] (5.05,0.7) to (5.05,-0.7);
\draw[dashed] (7.1,0.7) to (7.1,-0.7);
\draw[dashed] (8.35,0.7) to (8.35,-0.7);
\node (C1) at (2.25,-1.2) {Possible cuts of $D_{6n+3}\zeta^\fm(\{1,\bar2\}_k,1)$};
\end{tikzpicture}}
\end{equation}
By the same argument as for \eqref{fig:D_6n+3-bar21bar2} $\ncic{1}=0${} by anti-symmetry
and $\ncic{2}+\ncic{3}=0${} by path reversal, except for the cut $\ncic{3}\,${} starting
at the initial 0. Hence
\begin{equation} \label{equ:D_6n+3-1bar2k}
D_{6n+3}\zeta^\fm(\{1,\bar2\}_k,1)
=\zeta^\fl(\{1,\bar2\}_{2n+1}) \ot \zeta^\fm(\{1,\bar2\}_{k-2n-1},1).
\end{equation}

We note that there are two kinds of cuts for $D_r \zeta^\fm(3_{\ga},1,3_{\gb})$:
(i) those that don't involve the component $\rho(1)$ and (ii) those that do.
By periodicity, each type (ii) cut starts and ends with the same number so that it vanishes.
Further, as shown in \eqref{fig:D_6n+3-3a23bType(i)} type (i) has only one nontrivial contribution
\begin{equation}\label{equ:type(i)Contr313}
\zeta^\fl(3_{2n+1})\ot \zeta^\fm(3_{\ga-2n-1},1,3_{\gb}).
\end{equation}
Thus,
\begin{align*}
 -3\sum_{\ga+\gb=k, 2\nmid \gb}  D_{6n+3}\zeta^\fm(3_{\ga},1,3_{\gb})
 =&\, \zeta^\fl(3_{2n+1})\ot \left(-3\sum_{\ga+\gb=k, 2\nmid \gb}\zeta^\fm(3_{\ga-2n-1},1,3_{\gb})\right)\\
= &\,2^{3(k-2n-1)-1}  \zeta^\fl(3_{2n+1})\ot \zeta^\fm(\{1,\bar2\}_{k-2n-1},1) \\
= &\,2^{3k-1}  \zeta^\fl(\{1,\bar2\}_{2n+1})\ot \zeta^\fm(\{1,\bar2\}_{k-2n-1},1)
\end{align*}
by inductive assumption and \eqref{equ:1bar2Motivic}. Comparing to \eqref{equ:D_6n+3-1bar2k} we see that
\eqref{equ:D_r1bar2k1} holds for $r=6n+3$.

\subsubsection{$r=6n+5$}\label{sec:D_6n+5-1bar2k1}
It is obvious that the left-hand side of \eqref{equ:D_r1bar2k1} vanishes
for $r=6n+5$ since $\rho(\{1,\bar2\}_k,1)$
have period 6 so that every cut of $D_{6n+5}$ starts and ends with the same number.

Similar to the case $r=6n+3$, there are two kinds of cuts for $D_{6n+5}\zeta^\fm(3_{\ga},1,3_{\gb})$:
(i) those involve the component $\rho(1)$ and (ii) those don't. All cuts in (ii)
clearly vanish by periodicity. For type (ii) cuts we consider the following picture:

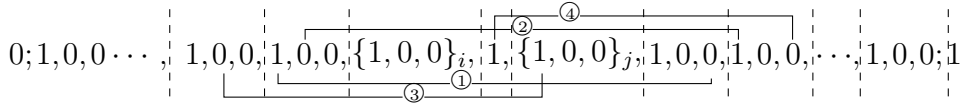
\begin{tikzpicture}[scale=0.9]
\node (A13) at (-3,0) {$0;$};
\node (A4) at (-2.6,0) {$1,$};
\node (A2) at (-2.2,0) {$0,$};
\node (A3) at (-1.4,0) {$0\cdots,$};
\draw[dashed] (-0.8,0.7) to (-0.8,-0.7);
\node (A4) at (-0.4,0) {$1,$};
\node (A6) at (-0.0,0) {$0,$};
\node (A5) at (0.4,0) {$0,$};
\draw[dashed] (0.6,0.7) to (0.6,-0.7);
\node (A8) at (0.85,0) {$1,$};
\node (A8) at (1.25,0) {$0,$};
\node (A9) at (1.65,0) {$0,$};
\node (A9) at (2.8,0.04) {$\{1,0,0\}_i,$};
\node (A10) at (4.05,0) {$1,$};
\node (A9) at (5.25,0.04) {$\{1,0,0\}_j,$};
\node (A12) at (6.45,0) {$1,$};
\node (A13) at (6.85,0) {$0,$};
\node (A13) at (7.25,0) {$0,$};
\node (A12) at (7.65,0) {$1,$};
\node (A13) at (8.05,0) {$0,$};
\node (A13) at (8.85,0) {$0,\cdots\!,$};
\node (A14) at (9.65,0) {$1,$};
\node (A14) at (10.05,0) {$0,$};
\node (A14) at (10.45,0) {$0;$};
\node (A15) at (10.85,0) {$1\phantom{,}$};
\node (C1) at (3.5,-0.4) {${}$};
\node (D2) at (4.4,0.4) {${}$};
\node (D3) at (2.8,-0.6) {${}$};
\node (D4) at (5.1,0.6) {${}$};
\draw (0.8,-0.25) to (0.8,-0.4) to (C1) node {$\cic{1}$} to (7.2,-0.4)  to (7.2,-0.325);
\draw (1.2,0.25) to (1.2,0.4) to (D2) node {$\cic{2}$} to (7.6,0.4)  to (7.6,0.25);
\draw (0,-0.25) to (0,-0.6) to (D3) node {$\cic{3}$} to (4.7,-0.6)  to (4.7,-0.25);
\draw (4.0,0.25) to (4.0,0.6) to (D4) node {$\cic{4}$} to (8.4,0.6)  to (8.4,0.25);
\draw[dashed] (1.85,0.7) to (1.85,-0.7);
\draw[dashed] (3.8,0.7) to (3.8,-0.7);
\draw[dashed] (4.25,0.7) to (4.25,-0.7);
\draw[dashed] (6.2,0.7) to (6.2,-0.7);
\draw[dashed] (7.45,0.7) to (7.45,-0.7);
\draw[dashed] (8.7,0.7) to (8.7,-0.7);
\draw[dashed] (9.4,0.7) to (9.4,-0.7);
\draw[dashed] (10.7,0.7) to (10.7,-0.7);
\node (C1) at (3.25,-1.2) {Type (ii) cuts of $D_{6n+5}\zeta^\fm(3_{\ga},1,3_{\gb}), i+j=2n$, starting before/at $\rho(1)$};
\end{tikzpicture}

We see immediately that cut  $\ncic{1}\,${} with index $(i,j)$ shown as in the picture is canceled by
the cut $\ncic{2}\,${} with index $(j,i)$ by path reversal.
The two special cases denoted by $\ncic{3}\,${} and $\ncic{4}\,${} with $i=j=2n$ cancel each other by path reversal.
Thus both sides of \eqref{equ:D_r1bar2k1} vanish if $r=6n+5$.

\subsubsection{$r=6n+7$}\label{sec:D_6n+7-1bar2k1}
By attaching one more 1 to the right end of the picture \eqref{fig:r=6n+7-bar21}  and \eqref{fig:r=6n+7-333}
exactly the same argument shows that $D_{6n+7}\zeta^\fm(\{1,\bar2\}_k,1)=0$
and all cuts for $D_{6n+5}\zeta^\fm(3_{\ga},1,3_{\gb})$ that do not involve $\rho(1)$ cancel each other.

For cuts of $D_{6n+5}\zeta^\fm(3_{\ga},1,3_{\gb})$ that involve the component $\rho(1)$ we consider the following picture:

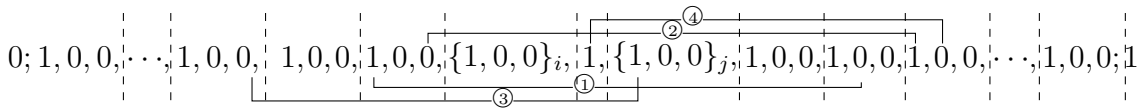
\begin{tikzpicture}[scale=0.9]
\node (A3) at (-2.75,0) {$0;1,0,0,\cdots\!,1,0,0,$};
\draw[dashed] (-2.9,0.7) to (-2.9,-0.7);
\draw[dashed] (-2.2,0.7) to (-2.2,-0.7);
\draw[dashed] (-0.8,0.7) to (-0.8,-0.7);
\node (A4) at (-0.4,0) {$1,$};
\node (A6) at (-0.0,0) {$0,$};
\node (A5) at (0.4,0) {$0,$};
\draw[dashed] (0.6,0.7) to (0.6,-0.7);
\node (A8) at (0.85,0) {$1,$};
\node (A8) at (1.25,0) {$0,$};
\node (A9) at (1.65,0) {$0,$};
\node (A9) at (2.8,0.04) {$\{1,0,0\}_i,$};
\node (A10) at (4.05,0) {$1,$};
\node (A9) at (5.25,0.04) {$\{1,0,0\}_j,$};
\node (A12) at (6.45,0) {$1,$};
\node (A13) at (6.85,0) {$0,$};
\node (A13) at (7.25,0) {$0,$};
\node (A12) at (7.65,0) {$1,$};
\node (A13) at (8.05,0) {$0,$};
\node (A13) at (8.45,0) {$0,$};
\node (A12) at (8.85,0) {$1,$};
\node (A13) at (9.25,0) {$0,$};
\node (A13) at (10.05,0) {$0,\cdots\!,$};
\node (A14) at (10.85,0) {$1,$};
\node (A14) at (11.25,0) {$0,$};
\node (A14) at (11.65,0) {$0;$};
\node (A15) at (12.05,0) {$1\phantom{,}$};

\node (C1) at (3.9,-0.4) {${}$};
\node (D2) at (5.2,0.4) {${}$};
\node (D3) at (2.7,-0.6) {${}$};
\node (D4) at (5.5,0.6) {${}$};
\draw (0.8,-0.25) to (0.8,-0.4) to (C1) node {$\cic{1}$} to (8,-0.4)  to (8,-0.325);
\draw (1.6,0.25) to (1.6,0.4) to (D2) node {$\cic{2}$} to (8.8,0.4)  to (8.8,0.25);
\draw (-1,-0.25) to (-1,-0.6) to (D3) node {$\cic{3}$} to (4.7,-0.6)  to (4.7,-0.25);
\draw (4.0,0.25) to (4.0,0.6) to (D4) node {$\cic{4}$} to (9.2,0.6)  to (9.2,0.25);
\draw[dashed] (1.85,0.7) to (1.85,-0.7);
\draw[dashed] (3.8,0.7) to (3.8,-0.7);
\draw[dashed] (4.25,0.7) to (4.25,-0.7);
\draw[dashed] (6.2,0.7) to (6.2,-0.7);
\draw[dashed] (7.45,0.7) to (7.45,-0.7);
\draw[dashed] (8.65,0.7) to (8.65,-0.7);
\draw[dashed] (9.9,0.7) to (9.9,-0.7);
\draw[dashed] (10.63,0.7) to (10.63,-0.7);
\draw[dashed] (11.9,0.7) to (11.9,-0.7);
\node (C1) at (3.25,-1.2) {Type (ii) cuts of $D_{6n+7}\zeta^\fm(3_{\ga},1,3_{\gb}), i+j=2n$, starting before/at $\rho(1)$};
\end{tikzpicture}

We see immediately that cut  $\ncic{1}\,${} with index $(i,j)$ shown as in the picture is canceled by
the cut $\ncic{2}\,${} with index $(j,i)$ by path reversal.
The two special cases denoted by $\ncic{3}\,${} and $\ncic{4}\,${} with $i=j=2n$ cancel each other by path reversal.
Thus both sides of \eqref{equ:D_r1bar2k1} vanish if $r=6n+7$.

Combining all the results in \S\ref{sec:D_6n+3-1bar2k1}-\S\ref{sec:D_6n+7-1bar2k1}, we see that
by Lemma~\ref{lem:D1} and Theorem~\ref{thm-Glanois}, there is some $e_k\in\Q$ such that
\begin{equation*}
 2^{3k-1}\zeta^\fm(\{1,\bar2\}_k,1)=e_k\zeta^\fm(3k+1) -3\sum_{\ga+\gb=k, 2\nmid \gb}\zeta^\fm(3_{\ga},1,3_{\gb}).
\end{equation*}
By applying the period map \eqref{equ:periodMap} we see that $e_k=0$. This shows that \eqref{equ:12bar1MotivicN}, 
namely, \eqref{equ:12bar1Motivic} holds for all $\ell\in\N$ which is the last statement in Theorem~\ref{thm:bar21bar2Motivic}.
Therefore we have completed the proof of Theorem~\ref{thm:bar21bar2Motivic}.

\bigskip
\noindent{\bf Funding.} Ce Xu is supported by the National Natural Science Foundation of China (Grant No. 12101008), the Natural Science Foundation of Anhui Province (Grant No. 2108085QA01) and the University Natural Science Research Project of Anhui Province (Grant No. KJ2020A0057). Jianqiang Zhao is supported by the Jacobs Prize from The Bishop's School.

\bigskip
\noindent{\bf Acknowledgments.} Both authors would like to thank Prof. F. Xu at the Capital Normal University and Prof. C. Bai at the Chern Institute of Mathematics for their warm hospitality where this work was initiated. They also want to thank the anonymous referee for carefully reading their manuscript and providing some very useful comments and suggestions which helped improve both the quality and the clarity of the paper.

\bigskip

\noindent{\bf Declaration of interest.} The authors have no competing interests to declare.

\end{document}